\documentclass[11pt, a4paper]{amsart}
\usepackage[dvips]{epsfig}
\usepackage{amsmath}
\usepackage{amssymb}
\usepackage{amsfonts}
\usepackage{amsthm}
\usepackage{amsbsy}
\usepackage{amsgen}
\usepackage{amscd}
\usepackage{amsopn}
\usepackage{amstext}
\usepackage{amsxtra}
\usepackage{enumerate}
\usepackage{dsfont}
\usepackage{hyperref}
\usepackage{lineno}
\usepackage{marginnote}
\usepackage{verbatim}
\usepackage{color}

\usepackage{times, graphicx}
\usepackage{eso-pic}
\usepackage{lastpage}

\numberwithin{equation}{section}

\newtheorem{theorem}{Theorem}[section]
\newtheorem{proposition}[theorem]{Proposition}
\newtheorem{lemma}[theorem]{Lemma}

\theoremstyle{definition}
\newtheorem{definition}[theorem]{Definition}
\newtheorem{notation}[theorem]{Notation}
\newtheorem{remark}[theorem]{Remark}
\newtheorem{conjecture}[theorem]{Conjecture}

\newcommand{\supp}{\operatorname{supp}}
\newcommand{\vol}{\operatorname{Vol}}

\newcommand{\prob}{\mathcal{P}}

\newcommand{\Dc}{\mathcal{D}}
\newcommand{\Gc}{\mathcal{G}}
\newcommand{\Tc}{\mathcal{T}}
\newcommand{\Sc}{\mathcal{S}}

\newcommand{\gfr}{\mathfrak{g}}

\newcommand{\nod}{\mathcal{N}}

\newcommand {\E} {\mathbb{E}}
\newcommand {\M} {\mathcal{M}}
\newcommand {\R} {\mathbb{R}}

\newcommand {\Pc} {\mathcal{P}}

\newcommand {\Ac} {\mathcal{A}}

\newcommand {\Zc} {\mathcal{Z}}

\begin{document}


\title[Volume distribution of nodal domains]{Volume distribution of nodal domains of random band-limited functions}

\author{Dmitry Beliaev and Igor Wigman}

\begin{abstract}
We study the volume distribution of nodal domains of random band-limited functions on generic manifolds,
and find that in the high energy limit a typical instance obeys a deterministic universal law,
independent of the manifold. Some of the basic qualitative properties of this law,
such as its support, monotonicity and continuity of the cumulative probability function, are established.
\end{abstract}

\maketitle

\tableofcontents

\section{Introduction}

A conjecture of M. Berry motivates our interest in the structure of nodal domains of random plane waves
and closely related spherical harmonics.
It turns out that a lot of techniques could be extended to more general ensembles of random functions that
we will discuss later on. In this paper we discuss the distribution of nodal domains' areas and length of their boundaries;
we are following the footsteps of Nazarov and Sodin ~\cite{NaSo09,NaSo15}, who developed some novel techniques to study the total
number of nodal domains of smooth fields, and Sarnak-Wigman ~\cite{SW} who extended their tools to study finer questions
of counting nodal domains of a given topological type and their mutual positions (``nestings").

Let $(\M,g)$ be a compact smooth Riemannian $n$-manifold.
For a smooth function $f:\M\rightarrow\R$ we define $\Omega(f)$ to be the collection of all nodal domains of $f$,
and for $t>0$ denote by $\nod(f;t)$
the number of nodal domains $\omega\in\Omega(f)$ of volume
$$
\vol(\omega)<t,
$$
where $\vol=\vol_n$ is the $n$-dimensional volume on $\M$.
We also define
\begin{equation*}
\nod(f) =  \nod(f;\infty) =
|\Omega(f)|
\end{equation*}
to be the total number of nodal domains of $f$.
In this paper we will investigate the behaviour of $N(f;t)$ for several classes of random functions $f$.

Before introducing the most general result we would like to discuss one particular case
(whose scaling limit is Berry's random monochromatic waves) which
is easy to explain and is representative in the proof of the general result.
It is well-known that the space of spherical harmonics of degree $l$ is of dimension $2l+1$; let $\{\phi_i\}_{i=0\ldots 2l+1}$
be its arbitrary $L^2$-orthonormal basis. Define a random Gaussian spherical harmonic
\begin{equation}
\label{eq:fl harm def}
f_l=\sqrt{\frac{1}{2l+1}}\sum\limits_{i=1}^{2l+1} c_i \phi_i,
\end{equation}
where $c_i$ are i.i.d.
standard Gaussian variables.
The normalizing constant is chosen so that $$\E[|f(z)|^2]=1$$ for every $z\in\Sc^{2}$.


For the total number of nodal domains of $f_{l}$
Nazarov and Sodin ~\cite{NaSo09,NaSo15} proved that there exists a constant $c_{0}=c(2,1)>0$ (``Universal Nazarov-Sodin
Constant") so that
\begin{equation}
\label{eq:NS result spher harm}
\E\left[\left| \frac{\nod(f_{l})}{l^{2}} -  4\pi c_{0} \right| \right]\rightarrow 0,
\end{equation}
i.e. that $\frac{\nod(f_{l})}{l^{2}}$ converges to the constant
$4\pi\cdot c_{0}>0$ {\em in mean}, where $4\pi=\vol_{2}(\Sc^{2})$ is the surface area of the unit $2$-sphere.
Our first principal result refines this.

\begin{theorem}
\label{thm:limit exist spher harm}

Let $f_l$ be the random spherical harmonic of degree $l$. Then the following holds:

\begin{enumerate}
\item
There exists a monotone non-decreasing function
$$
\Psi=\Psi_{2,1}:(0,\infty)\rightarrow\R_{+ },
$$
continuous outside a countable set of jumps $\Tc_{0}=\Tc_{0;2;1} = \{t_{k}\}_{k=1}^{\infty}$,
so that for all $t\in \R\setminus\Tc_{0}$ we have
\begin{equation*}
\E\left[\left|  \frac{\nod(f_{l};t/l^{2})}{\nod(f_{l})} - \Psi(t) \right|\right] \rightarrow 0.
\end{equation*}

(Notation $\Psi_{2,1}$ and $\Tc_{0;2;1}$ will be clear from the formulation of Theorem \ref{thm:limit exist band lim}.)

\item

Let
$$
t_0=\pi j_{0,1}^2=18.168\ldots
$$
where $j_{0,1}\approx 2.4048$ is the first zero of the Bessel function $J_0$.
Then the function $\Psi$ defined above vanishes on $[0,t_{0})$, and is strictly increasing for $t>t_{0}$.

\end{enumerate}

\end{theorem}

It seems unlikely that there exist ``distinguished" numbers $t>0$ that accumulate positive proportions of nodal domain areas.
We thereupon conjecture that the limiting distribution has no atoms (cf. Conjecture \ref{conj:Psi lim dist gen band lim}):

\begin{conjecture}
\label{conj:Psi lim dist smooth spher}
Let $\Psi$ be the function prescribed by Theorem \ref{thm:limit exist spher harm}.

\begin{enumerate}

\item The set of jumps of $\Psi$ is empty $\Tc_{0} = \emptyset$, i.e. $\Psi$ is continuous.

\item The function $\Psi$ is everywhere differentiable on $(t_0,\infty)$ and its derivative
(the probability density of the limiting distribution of nodal areas) is strictly positive.

\end{enumerate}

\end{conjecture}

Let us make a few remarks on Theorem \ref{conj:Psi lim dist smooth spher}. The spherical harmonics have a natural scale $1/l$, hence the natural area
scaling is $t/l^2$;
the area of a typical nodal domain is of order of magnitude $\frac{1}{l^2}$.
Since the spherical harmonics are eigenfunctions of the spherical Laplacian,
there is a lower bound on the nodal domains area (Faber-Krahn inequality), and hence
the fact that $\Psi(t)=0$ for $t<t_0$ is deterministic; in fact, every nodal domain contains a disc of diameter comparable to $1/l$.
One may formulate and prove similar results on the distribution of lengths of nodal domains boundaries, called the {\em nodal components} of $f_{l}$
(see also Section \ref{sec: final remarks}). 

All the other principal results of this paper are in the same spirit as Theorem \ref{thm:limit exist spher harm} but in different,
less specialized, settings. The two main settings are: Euclidian (or ``scale invariant") case, and band-limited ensembles on
manifolds; one recovers the former as scaling limits of the band-limited ensembles around every point of the given manifold.
Below we briefly describe the various settings of random ensembles of functions.

\subsection{Euclidian random fields}
Here we are interested in centred Gaussian functions $F:\R^n\rightarrow \R$; it is a well known fact
(Kolmogorov's Theorem) that the distribution of a centred Gaussian field is completely determined by its covariance kernel
$$
K(x,y)=\E [F(x)F(y)].
$$
We will be interested in {\em isotropic} fields i.e. the fields such that
$$K(x,y)=K(|x-y|),$$ which means that $F$ is invariant under translations and rotations. From now on we assume that all our fields in $\R^n$ are isotropic,
moreover, we will also assume that $F$ is normalized so that  $K(0)=\E [F^2(x)] =1.$

It is known that the such covariance kernel $K$ can be written as  Fourier transform of a measure $d\rho$, called the {\em spectral measure}. In many
cases it is more convenient to describe the field $F$ in term of its spectral measure instead of the covariance kernel. Given the spectral measure, there is an alternative way of constructing the random function $F$. It can be constructed as the Gaussian vector in the Hilbert space $\mathcal{H}$ which is Fourier image of the the symmetric space $L^2_{\mathrm{sym}}(d\rho)$. In particular this means that if $\phi_k$ is an orthonormal basis in $\mathcal{H}$, then $F=\sum \zeta_k \phi_k$, where $\zeta_k$ are i.i.d Gaussian random variables.
One of the most important examples which motivates the paper is the random plane wave:

\begin{definition}
The {\em random plane wave} or the {\em monochromatic wave} with energy $E=k^2$ is the centred Gaussian field on $\R^{2}$
with the covariance kernel $$K(x,y)=J_0(k|x-y|),$$ where $J_0$ is the zeroth Bessel function.
\end{definition}

It is possible to check that the random plane wave is a solution of Helmholz equation
\begin{equation}
\label{eq:Helmholz}
\Delta f+k^2 f=0.
\end{equation}
One may think of the random plane wave as a ``random''
solution of \eqref{eq:Helmholz}. One possible way to quantify this statement is to notice that the functions
$J_{|n|}(kr)e^{i n \theta}$ form an orthonormal basis in the space of all $L^2$ solutions.
Hence it is natural to write its random Gaussian solution as
\begin{equation}
\label{eg: def RPW}
F(x)=\sum_{n=-\infty}^\infty c_n J_{|n|}( kr)e^{i n \theta},
\end{equation}
where the coefficients $c_n$ are standard Gaussians, independent save to the condition $c_{-n}=\bar{c}_n$. More details about the construction of the
random plane wave and its connection with the random spherical harmonic could be found
in ~\cite{NaSo09}.

A direct computation shows that the covariance kernel of this function is indeed $K(x,y)=J_0(k|x-y|)$. The spectral measure of $F$ is the normalized
Lebesgue measure on the  circle of radius $k$.
Since plane waves with different values of $k$ differ by the scaling, it is natural to fix $k=1$.
The description in terms of the spectral measure has a natural generalization to the higher dimensions:

\begin{definition}
The random plane wave in $\R^n$ is the Gaussian field whose spectral measure is the normalized $(n-1)$-dimensional Lebesgue measure on $S^{n-1}\subset
\R^n$.
\end{definition}

In Theorem \ref{thm:scal invar exist} below we will show that under relatively mild
conditions on the spectral measure, an analogue of Theorem \ref{thm:limit exist spher harm} holds
for $F$. Theorem \ref{thm:scal invar exist} will be proved in great generality, though the most important, relevant for Theorem \ref{thm:limit exist
spher harm}, is the case of the random monochromatic plane-wave, of significant importance by itself.

\subsection{Band-limited functions}

Let $(\M,g)$ be a compact Riemanian $n$-manifold,
then the eigenfunctions of the Laplacian $\{\phi_i(x)\}_{i\ge 1}$ form an orthonormal basis of $L^{2}(\M)$.
We denote the square roots of eigenvalues by $0=t_0\le t_1\le t_2 \ldots $
i.e. satisfying
$$
\Delta \phi_i+t_i^2\phi_i=0.
$$

Fix $\alpha\in [0,1)$ and define the $\alpha$-band-limited functions
\begin{equation}
\label{eq:f band lim def}
f(x) =f_{\alpha;T}(x)= \sum\limits_{\alpha T \le t_{j} \le T} c_{j}\phi_{j}(x).
\end{equation}
where $c_{j}$ are independent real Gaussian variables of mean $0$ and variance $1$. For $\alpha=1$ we define $f_{1;T}$ by
$$
f(x) =f_{1;T}(x)= \sum\limits_{T-\eta(T) \le t_{j} \le T} c_{j}\phi_{j}(x),
$$
where $\eta(T)$ is a growing function such that $\eta(T)=o(T)$.

The random spherical harmonic \eqref{eq:fl harm def} defined above is a $\alpha=1$ band-limited function on the unit sphere $\Sc^{2}$ with $\alpha=1$ and $\eta(T)=O(T^{1/2})$.
For the total number of nodal domains of the band limited functions 
Nazarov and Sodin proved that for every $\alpha\in [0,1]$, $n\ge 2$ there exist a constant $c(n,\alpha)>0$
(``Nazarov-Sodin constant"\footnote{Note the different normalization as compared to ~\cite{SW}}) satisfying
\begin{equation}
\label{eq:NS result band lim}
\E\left[\left| \frac{\nod_{\Omega}(f)}{T^{n}} -
c(n,\alpha) \cdot \vol_{n}(\M) \right| \right]\rightarrow 0.
\end{equation}
Sarnak and Wigman ~\cite{SW} refined the latter result \eqref{eq:NS result band lim} for counting the number of nodal domains
(or components) of $f$ of a given topological class; they also found an elegant way to formulate it in terms of 
convergence of random probability measures consolidating all topological types into a universal deterministic probability measure that 
conserves all the topologies.
We will show below that the result similar to Theorem \ref{thm:limit exist spher harm} holds for band-limited functions.

\begin{theorem}
\label{thm:limit exist band lim}

Let $f=f_{\alpha;T}$ be a band-limited function on some $n$-dimensional compact manifold $\M$.
There exists a monotone non-decreasing function
$$
\Psi=\Psi_{n;\alpha}:(0,\infty)\rightarrow\R_{+ },
$$
continuous outside a countable set of jumps $\Tc_{0} = \Tc_{0;n;\alpha}=\{t_{k}\}_{k=1}^{\infty}$,
so that for all $t\in \R\setminus\Tc_{0}$ we have
\begin{equation}
\label{eq:lim exist band lim}
\E\left[\left|  \frac{\nod_{\Omega}(f;t/T^{n/2})}{\nod_{\Omega}(f)} - \Psi(t) \right|\right] \rightarrow 0.
\end{equation}

\end{theorem}

\noindent Importantly, the function $\Psi$ prescribed by Theorem \ref{thm:limit exist band lim} does not depend on the manifold.

\subsection{Some properties of the limiting distribution $\Psi_{n;\alpha}$}

Here we investigate the most basic properties of $\Psi_{n;\alpha}$, i.e. its monotonicity.
We need to distinguish between $\alpha<1$, where the corresponding distribution function $\Psi_{n;\alpha}$
is strictly positive and increasing everywhere (Theorem \ref{thm:alpha<1 Psi increas}), and $\alpha=1$, where the behaviour of the function
$\Psi_{n;1}$ is more complicated (Theorem \ref{thm:alpha=1 Psi increas t>t0}).

\begin{theorem}
\label{thm:alpha<1 Psi increas}
For every $n\ge 2$, $\alpha < 1$ the function $\Psi_{n;\alpha}$
as above is strictly increasing on $\R_{+}$.
\end{theorem}

\begin{theorem}
\label{thm:alpha=1 Psi increas t>t0}
For every $n\ge 2$ the function $\Psi_{n;1}$ vanishes on $[0,t_{0}]$ where
$$
t_{0}=t_{0}(n)=\frac{\pi^{n/2}}{\Gamma(n/2+1)}j^n_{n/2-1,1}
$$
is the volume of the ball of radius $j_{n/2-1,1}$ -- the first zero of the Bessel function $J_{n/2-1}$.
Moreover, $\Psi_{n;1}$ is strictly increasing on $(t_0,\infty)$.
\end{theorem}

Motivated by similar arguments to Conjecture \ref{conj:Psi lim dist smooth spher} it is only natural to conjecture the following:

\begin{conjecture}
\label{conj:Psi lim dist gen band lim}


The function $\Psi_{n;\alpha}$ is continuous, everywhere differentiable. For $\alpha<1$ the derivative $\frac{d\Psi_{n;\alpha}(t)}{dt}>0$
is everywhere positive, whereas for $\alpha=1$ the same holds for $t>t_{0}$.

\end{conjecture}

\subsection{Main ideas and the plan of the paper}

Our general stategy is similar to ~\cite{NaSo15} and ~\cite{SW}; for the start we give some necessary background about the semiclassical behaviour of band-limited functions. Roughly speaking, we show that on a scale which is small but larger than $1/T$, a rescaled version $f_{\alpha;T}$ of the band-limited function could be approximated by its scaling limit $\gfr_{n,\alpha}$ defined on the tangent space. Two important facts are that $\gfr_{n,\alpha}$ is {\em universal}, that is independent of the manifold or a particular point on the manifold, and that the covariance kernel (or the spectral measure) of the scaling limit $\gfr_{n,\alpha}$ is known explicitly. This will allow us to prove the main results for $\gfr_{n,\alpha}$, and then use the convergence to prove the results for the band-limited functions.

In Section \ref{subsec: Kac-Rice} we formulate several Kac-Rice type results that will yield universal upper bounds on various quantities like the number of nodal domain or the number of nodal components.
In the next Section \ref{sec:scal invar} we discuss the behaviour of Gaussian fields on $\R^n$. We formulate and prove
Theorem \ref{thm:scal invar exist} stating that under very mild conditions on a random field, satisfied by $\gfr$, a Euclidean version
of \ref{thm:limit exist band lim} holds. The proof is based on {\em a priori} upper bounds given by Kac-Rice and ergodic theorem which gives the existence of the scaling limit of the volume distribution. At the end of the Section \ref{sec:scal invar} we prove two theorems about the distribution function $\Psi_{n,\alpha}$: Theorems \ref{thm:scal invar increase gen}  and \ref{thm:scal invar increase RWM}. They are Euclidean analogues of Theorems \ref{thm:alpha<1 Psi increas} and \ref{thm:alpha=1 Psi increas t>t0}. The first one covers the generic case $\alpha<1$ and the second one
applies on $\alpha=1$.

The proofs of both Theorems \ref{thm:scal invar increase gen} and \ref{thm:scal invar increase RWM} are rather similar. Our first goal is to show that there exists a {\em deterministic} function whose nodal domain containing the origin is of almost the required volume. Our second goalis to show that for all functions that are close to the postulated one in $C^1$-norm have nodal domains of similar volume (see Lemma \ref{lem:vol pert}). Finally we claim that the latter happens with positive probability. The translation invariance of the underlying random fields yields that, as the claimed result holds near origin with positive probability, it holds with positive density.

In Section \ref{sec:local results} we prove that since functions $f_{\alpha;T}$ and $\gfr_{n,\alpha}$ are close to each other (after coupling and rescaling), their respective numbers of nodal domains of restricted volume is close with high probability (this is quantified in Proposition \ref{prop:bi(R-1,Fx)<=bi(R/L,fL)<=bi(R+1,Fx)}).
In Section \ref{sec: global results} we prove all the main theorems of the paper. The proofs are based on {\em  semi-locality} of nodal domains. This means that most of the nodal domains of $f_{\alpha; T}$ are not too small nor too long, i.e. that the semi-local approximation by $\gfr_{n,\alpha}$ captures most of the nodal domains of $f_{\alpha;T}$. This allows to infer the results on $f_{\alpha;T}$ from the analogous results on $\gfr_{n,\alpha}$.
Finally, in Section \ref{sec: final remarks} we make some final remarks about the proof, in particular we explain that with some minor modification our methods imply similar results for the distribution of the surface volume of the nodal components or even for the joint distribution of the volumes of nodal domains and boundary volumes.

\subsection{Acknowledgements}

We would like to thank P. Sarnak for raising the question of nodal volumes and his interest in our work,
and M. Sodin for many useful discussions. The first author was partially funded by Engineering \& Physical Sciences
Research Council (EPSRC) Fellowship ref. EP/M002896/1.
The research leading to these results has received funding from the
European Research Council under the European Union's Seventh
Framework Programme (FP7/2007-2013), ERC grant agreement
n$^{\text{o}}$ 335141 (I.W.)

\section{Necessary background}

\subsection{Semiclassical properties of band-limited functions and their scaling limits}
\label{sec:semiclassical limit}

In this section we introduce a few facts about band-limited functions. The results are stated without proofs, for more detailed discussion we refer the readers to \cite{SW}[Section 2.1] and references therein.

For the band-limited function $f_{\alpha;T}$ we have the covariance function
$$
K_{\alpha; T}(x,y):= \E\left[f_{\alpha;T}(x)f_{\alpha;T}(y)\right]=\sum\limits_{\alpha T \le t_{j} \le T} \phi_{j}(x)\phi_{j}(y).
$$
The following semiclassical approximation, universal w.r.t. $\M$ holds (see ~\cite[Section 2.1]{SW} with case $\alpha=1$ due to
Canzani-Hanin ~\cite{CH,CH2}):
\begin{equation}
\label{eq:Kalpha asymp}
\widetilde{K_{\alpha}} (T;x,y) := \frac{1}{D_{\alpha}(T)}K_{\alpha} (T;x,y) = B_{n,\alpha}(T\cdot d(x,y)) + O\left(T^{-1}\right),
\end{equation}
uniformly for $x,y\in\M$, where $d(x,y)$ is the (geodesic) distance in $\M$ between $x$ and $y$,
\begin{equation*}
D_{\alpha}(T)=\frac{1}{\vol(\M)} \int\limits_{\M}K_{\alpha}(T;x,x)d\vol(x),
\end{equation*}
and for $w\in\R^{n}$
\begin{equation}
\label{eq:Bnalp clean}
B_{n,\alpha}(w) = B_{n,\alpha}(|w|) = \frac{1}{|A_{\alpha}|} \int\limits_{A_{\alpha}}e^{2\pi i\langle w,\xi  \rangle} d\xi
\end{equation}
where $ A_{\alpha}=\left\{ w:\: \alpha\le |w|\le 1 \right\}$ and in the case $\alpha=1$ the $n$-dimensional measure $d\xi$ is replaced by $(n-1)$-dimensional surface measure on the unit sphere.

We may differentiate both sides of \eqref{eq:Kalpha asymp} to obtain asymptotic expression for {\em finitely} many derivatives
of $K_{\alpha}$.
By appropriately normalizing $f_{\alpha}$ we may assume w.l.o.g that $\widetilde{K_{\alpha}}$ is the covariance of $f_{\alpha}$
and we will neglect this difference from this point on.

Around each point $x\in \M$ we define
the scaled random fields $f_{x;T}$ (we drop $\alpha$ from notations) on a big ball (after scaling) lying on the Euclidian tangent space
$\R^{n}\cong T_{x}\M$ with the use of the exponential
map and an isometry $I_{x}:\R^{n}\rightarrow T_{x}\M$, $\Phi_{x}=\exp_{x}\circ I_{x}:\R^{n}\rightarrow \M$ via
\begin{equation}
\label{eq:fxL def}
f_{x;T}(u) := f_{T}(\Phi_{x}(u/T)),
\end{equation}
with the covariance function
\begin{equation*}
K_{x;T}(u,v) := \E[f_{x;T}(u)\cdot f_{x;T}(v)] = K_{T}(\Phi_{x}(u/T),\Phi_{x}(v/T)).
\end{equation*}

Observe that locally $\Phi_{x}$ is almost an isometry:
for each $\xi$ there exists $r_{0}=r_{0}(\xi)$ such that if $\mathcal{D}\subseteq B_{x}(r_{0})$
is a smooth domain in $\M$, then
\begin{equation}
\label{eq:|vol(D)-vol(Phix(d))|<eps}
|\vol_{T_{x}}(\mathcal{D}) -\vol_{\M}(\Phi_{x}(\mathcal{D}))|< \xi,
\end{equation}
uniformly w.r.t. $x\in \M$ ($r_{0}$ is assumed to be sufficiently small so that the exponential map is
$1-1$). By the scaling \eqref{eq:fxL def} we obviously have
\begin{equation*}
\nod\left(f,\frac{t}{T^{n}};x,\frac{R}{T}\right) \approx \nod\left(f_{x;T},t;x,R\right).
\end{equation*}
The precise meaning will be given by \eqref{eq:nod fxL<=nod fL<=nodfxL}. This means that studying $f_{\alpha;T}$ and $f_{x;T}$ is essentially equivalent. Finally, from \eqref{eq:Kalpha asymp} we see that the covariance kernel of $f_{x,T}$ converges to
$$
r_{n,\alpha}(u,v)=B_{n,\alpha}(|u-v|).
$$
This suggests to define the local scaling limit $\gfr_{n,\alpha}$ to be a Gaussian function in $\R^n$ with this covariance kernel. Alternatively, it could be defined by its spectral measure which by \eqref{eq:Bnalp clean} is the normalized Lebesgue measure on $A_\alpha$ (or the normalized surface area on $A_1$ for $\alpha=1$). It is important to point out that the scaling limit is universal: it does not depend on $x$ or $\mathcal{M}$, but
only on $n$ and $\alpha$.

Functions $f_{x,T}$ and $\gfr_{n,\alpha}$ are defined on different probability spaces, but it is possible to couple them in such a way that they are close for large $T$. The precise meaning is given by the following lemma.

\begin{lemma}[~\cite{So12}, Lemma $4$]
\label{lem:E[|fxL-Fx|<alpha]}
Given $R>0$ and $b>0$, there exists $T_{0}=T_{0}(R,b)$ such that for all $T\ge T_{0}$ we have
\begin{equation*}
\E\left[\|f_{x;T}-\gfr_{n,\alpha}\|_{C^{1}(\overline{B}(2R))}\right] < b.
\end{equation*}
\end{lemma}

\subsection{The Kac-Rice premise}
\label{subsec: Kac-Rice}

In this section we collect a number of {\em local} results required below. The Kac-Rice formula is a powerful
tool for computing moments of local quantities, and in principle it is capable of expressing {\em any} moment
of a local quantity of a given random Gaussian field $F$
in terms of explicit, albeit complicated, Gaussian expectations, or, equivalently, in terms of the covariance function
$$r(x,y):=\E[F(x)\cdot F(y)].$$ In this manuscript we will only need upper bounds for expectations
of some more or less tricky local quantities; typically we will fix the random field $F$ over expanding
balls $B(R)$ as $R\rightarrow\infty$ with one exception, where we will require uniform control
w.r.t. both the spectral parameter $T\rightarrow\infty$ for the band-limited functions on geodesic balls,
and their centres (see Lemma \ref{lem:Kac-Rice band lim bnd unif} below). The proofs of these results will be omitted;
instead we refer the reader to other sources where these are fully proved.

Let $m\le n$, $F:\Dc\rightarrow\R^{m}$ be a smooth random field
on a domain $\Dc\subseteq \R^{n}$, and $\Zc(F;\overline{\Dc})$ be either the $(n-m)$-volume of $F^{-1}(0)$
(for $m<n$), or the number of the discrete zeros (for $m=n$). For example, if $H:\M\rightarrow \R$ is a random field
and $F=\nabla H|_{\Dc}:\M\rightarrow \R^{n}$ is its gradient restricted to a coordinate patch, then $\Zc(F,\overline{\Dc})$
counts the number of critical points of $H$ on $\Dc$.
We set $J_{F}(x)$ to be the (random) Jacobi matrix of $F$ at $x$,
and define the ``zero density" of $F$ at $x\in \Dc$ as the conditional Gaussian expectation
\begin{equation}
\label{eq:K1 density def}
K_{1}(x) = K_{1;F}(x)=\E[|\det J_{F}(x)| \big| F(x)=0].
\end{equation}

With the above notation the Kac-Rice formula (meta-theorem) states that, under some non-degeneracy condition on $F$,
$$\E[\Zc(F;\Dc)]=\int\limits_{\Dc}K_{1}(x)dx.$$ Concerning the sufficient conditions that guarantee that \eqref{eq:K1 density def} holds,
a variety of results is known ~\cite{Adler-Taylor,AW}. The following version of Kac-Rice merely requires the non-degeneracy
of the values of $F$ (vs. the non-degeneracy of $(F,J_{F}(x))$ in the appropriate sense, as in
some more classical sources), to our best knowledge, the mildest sufficient condition.

\begin{lemma}[Standard Kac-Rice ~\cite{AW}, Theorem $6.3$]
\label{lem:Kac-Rice precise}
Let $F:\Dc\rightarrow\R^{m}$ be an a.s. smooth Gaussian field, such that for every $x\in \Dc$ the distribution of
the random vector $F(x)\in \R^{m}$ is non-degenerate Gaussian. Then
\begin{equation}
\label{eq:E[ZH]=int K1}
\E[\Zc(F;\Dc)]=\int\limits_{\Dc}K_{1}(x)dx
\end{equation}
with the zero density $K_{1}(x)$ as in \eqref{eq:K1 density def}.
\end{lemma}

The following lemma is an
upper bound for either the number of critical points of a random field or its restriction
to a hypersphere as a result of an application of Kac-Rice on coordinate patches.

\begin{lemma}[~\cite{SW}, Corollary 2.3]
\label{lem:Kac Rice crit ball sphere}

Let $\Dc\subseteq\R^{m}$ be a domain and $F:\Dc\rightarrow\R$ an a.s. $C^{2}$-smooth stationary Gaussian random field,
such that for $x\in \Dc$ the distribution of $\nabla F(x)$ is non-degenerate Gaussian.

\begin{enumerate}

\item For $r>0$ let $\Ac(F;r)$ be the number of critical points of $F$ inside $B(r)\subseteq\Dc$. Then
$$\E[\Ac(F;r)] = O(\vol(B(r))),$$ where the constant involved in the `$O$'-notation depends on the law of $F$ only.

\item For $r>0$ let $\widetilde{\Ac}(F;r)$ be the number of critical points of
the restriction $F|_{\partial B(r)}$ of $F$ to the sphere $\partial B(r)\subseteq\Dc $. Then
$$\E[\widetilde{\Ac}(F;r)]= O(\vol(\partial B(r))),$$
where the constant involved in the `$O$'-notation depends on the law of $F$ only.

\end{enumerate}

\end{lemma}

The following estimate is the upper bound part of the (precise) Kac-Rice formula applied
to the band limited functions.

\begin{lemma}[~\cite{So12} Lemma $2$, ~\cite{SW} Lemma $7.8$]
\label{lem:Kac-Rice band lim bnd unif}
For $x\in\M$, $r>0$ let $\nod_{\Omega}(f_{\alpha;T};x,r)$ be the number of nodal domains of $f_{\alpha;T}$
entirely lying in $B(x,r)\subseteq\M$: the geodesic ball of radius $r$  centred at $x$.
Then
\begin{equation*}
\E[\nod_{\Omega}(f_{\alpha;T};x,r)]= O(r^{n}\cdot T^{n}),
\end{equation*}
with constant involved in the $`O'$-notation depending on $\M$ and $\alpha$ only.
\end{lemma}

\section{Distribution of nodal domain areas for Euclidian fields}

\label{sec:scal invar}

\subsection{Notation and statement of the main result on Euclidian fields}

\subsubsection{Notation and basic setup}

\label{sec:not bas set}

Let $f$ be a smooth function, $t>0$, and $R>0$.
We denote $\nod(f,t;R)$ to be
the number of domains $\omega\in \Omega(f)$ of $f$ lying entirely in $B(R)$ of volume
$$
\vol_{n}(\omega)\le t;
$$
note that
$$
\nod(f;R)=\nod(f,\infty;R)
$$
is the total number of nodal domains lying inside $B(R)$, as considered by Nazarov and Sodin \cite{NaSo09}.

We are interested in the asymptotic distribution\footnote{To make sense of the distribution of $\nod(F,\cdot;\cdot)$ it is essential to show that the latter is a {\em random variable},
i.e. a measurable function on the sample space. Fortunately, the proof of a similar statement, given in
~\cite[Appendix A]{SW}, is sufficiently robust to cover our case and all the other similar quantities of this manuscript;
from this point on we will neglect any issue of measurability.} of the nodal domain volumes,
that is, the asymptotic behaviour of $\nod(F,t;R)$ as $R\rightarrow\infty$, $t>0$ fixed.
We would like to establish the limit
\begin{equation}
\label{eq:Psi(t)=lim(nod(t))/nod}
\Psi_{\rho}(t) := \lim\limits_{R\rightarrow\infty}\frac{\nod(F,t;R)}{\nod(F;R)}
\end{equation}
in mean, and, moreover, that $\Psi_{\rho}(t)$ is a distribution function, i.e.
$$
\lim\limits_{t\rightarrow\infty}\Psi_{\rho}(t)=1.
$$
The latter will follow once having established
the former (in the proper sense) via the obvious {\em deterministic} upper
bound
$$
\nod(F;R)- \nod(F,t;R) \lesssim \frac{R^{n}}{t}
$$
for the number of domains of volume greater than $t$.

Following Nazarov and Sodin \cite{NaSo15} we assume that the spectral measure $d\rho$ of $F$ satisfies the following axioms:



\begin{enumerate}

\item [$(\rho 1)$] The measure $d\rho$ has no atoms (if and only if the action of the translations
is ergodic by Grenander-Fomin-Maruyama, see ~\cite[Theorem $3$]{So12}).

\item [$(\rho 2)$] For some $p>4$,
$$
\int\limits_{\R^{n}}|\lambda|^{p}d\rho(\lambda) < \infty
$$
(this ensures the a.s. smoothness of $F$).

\item [$(\rho 3)$] The spectral support $\supp \rho$ does not lie in a linear hyperplane. (This ensures that
the random Gaussian field, together with its gradient is not degenerate.)

\end{enumerate}

For this model Nazarov-Sodin ~\cite{NaSo15} proved that there exists a constant $c(\rho)\ge 0$ (the ``Nazarov-Sodin constant")
so that
\begin{equation}
\label{eq:tot numb asymp NS}
\frac{\nod(F;R)}{\vol B(R)} \rightarrow c(\rho)
\end{equation}
both in mean and a.s.

It is shown in \cite{NaSo15} that under the additional mild condition  the constant $c(\rho)$ is strictly positive. We don't want to discuss these
technicalities, so instead we will use the assumption

\begin{enumerate}
\item [$(\rho 4)$]
The Nazarov-Sodin constant $c(\rho)$ is positive.
\end{enumerate}

Sometimes we will need a stronger  axiom:

\begin{enumerate}
\item [$(\rho 4^{*})$]
The support of the spectral measure $\rho$ has  non-empty interior.
\end{enumerate}

\subsubsection{Existence of limiting distribution $\Psi_{\rho}$}

\begin{theorem}
\label{thm:scal invar exist}
Let $F:\R^{n}\rightarrow\R$ be an a.s. smooth random field whose spectral density $\rho$
satisfies the axioms $(\rho 1)-(\rho 3)$, then
$$
\frac{\nod(F,t;R)}{ \vol(B(R))}
$$
converges in mean as $R\to \infty$. If we additionally assume the axiom $(\rho 4)$, then the limit
\begin{equation*}
\lim\limits_{R\rightarrow\infty}\frac{\nod(F;R)}{ \vol(B(R))} = c(\rho)>0
\end{equation*}
does not vanish, so that we can define the normalized limit
\begin{equation}
\label{eq:bi/volB(R)->eta}
\Psi_{\rho}(t):=\lim\limits_{R\rightarrow\infty}\frac{\nod(F,t;R)}{c(\rho)\cdot \vol(B(R))}.
\end{equation}

\end{theorem}

Since the total number of nodal domains of $F$ lying inside $B(R)$ was proven to be asymptotic to
\begin{equation*}
\nod(F;R)\sim c(\rho)\cdot \vol B(R),
\end{equation*}
(see \eqref{eq:tot numb asymp NS}), \eqref{eq:bi/volB(R)->eta} may be equivalently read as
\begin{equation*}
\frac{\nod(F,t;R)}{\nod(F;R)}\rightarrow  \Psi_{\rho}(t),
\end{equation*}
(cf. \eqref{eq:Psi(t)=lim(nod(t))/nod}); this limit could be proven in mean, see the proof of Theorem \ref{thm:limit exist band lim}
in section \ref{sec:thm lim exist proof}.
Theorem \ref{thm:scal invar exist} in particular implies that for every $t>0$ the expected number
$\nod(F,t;R)$ of nodal domains of $F$ of volume at most $t$ lying in $B(R)$ is
\begin{equation}
\label{eq:E[bi(R)] sim Vol(B(R))}
\E[\nod(F,t;R)]= c(\rho)\cdot\Psi_{\rho}(t)\cdot\vol(B(R))(1+o_{R\rightarrow\infty}(1)),
\end{equation}
with concentration: for every $\epsilon>0$
\begin{equation}
\label{eq:concentration around mean}
\lim\limits_{R\rightarrow\infty} \prob \left\{ \left| \frac{\nod(F,t;R)}{\vol B(R)}  - c(\rho)\cdot \Psi_{\rho}(t)\right| > \epsilon  \right\}
=0,
\end{equation}
via Chebyshev's inequality.

\subsubsection{Some properties of $\Psi_{\rho}(t)$}

\begin{theorem}[Lower bound on domains in the generic case]
\label{thm:scal invar increase gen}

Assume that the spectral measure of $F$ satisfies axioms
$(\rho 1)-(\rho 3)$ and $(\rho 4^{*})$.
Then $\Psi_{\rho}(0)= 0$, and $\Psi_{\rho}(\cdot)$ is strictly increasing on $[0,\infty)$;
in particular for every $t>0$ we have $\Psi_{\rho}(t)> 0$.

\end{theorem}

For the random plane wave the situation only slightly differs from the above in that $\Psi_{RWM}$ vanishes up to a certain explicit threshold.

\begin{theorem}
\label{thm:scal invar increase RWM}

Let $F=F_{RPW}$ be the random plane wave in $\R^n$. As its spectral measure satisfies axioms $(\rho 1)-(\rho 4)$,
the function $\Psi_{RPW}=\Psi_{\rho}$ defined as in Theorem \ref{thm:scal invar exist} exists.
Define
$$
t_{0}=t_{0}(n)=\frac{\pi^{n/2}}{\Gamma(n/2+1)}j^n_{n/2-1,1}
$$
as in Theorem \ref{thm:alpha=1 Psi increas t>t0}. Then the following holds:

\begin{enumerate}

\item
\label{it:area(omega)>t0}

For every nodal domain $\omega$ of $F_{RPW}$ we have
$ \vol(\omega)\ge t_0,$
and hence, in particular,  we have
$$
\Psi_{RPW}(t) =0, \qquad  t<t_{0}
$$

\item
Moreover, for $n=2$ the function $\Psi$ is strictly increasing on $[t_0,\infty)$


\end{enumerate}

\end{theorem}

The remaining part of section \ref{sec:scal invar} is dedicated to the proofs of theorems
\ref{thm:scal invar exist}, \ref{thm:scal invar increase gen} and \ref{thm:scal invar increase RWM}.

\subsection{Integral-Geometric Sandwich}

Let $\Gamma\subseteq\R^{n}$ be a hypersurface (a curve for $n=2$).
For $u\in \R^{n}$, $r>0$ and $t>0$ we denote $\nod(\Gamma,t;u,r)$ the number
of domains of $\Gamma$ of $n$-dimensional volume bounded by $t$ lying entirely in $B_{u}(r)$ .
Similarly, define $\nod^{*}(\Gamma,t;u,r)$ by relaxing the condition to domains merely {\em intersecting}
$\overline{B}_{u}(r)$. We use the shortcuts
$$
\nod(g,t;u,r) :=\nod(g^{-1}(0),t;u,r),
$$
respectively
$$
\nod^{*}(g,t;u,r) :=\nod^{*}(g^{-1}(0),t;u,r),
$$
and $\nod(\cdot,\cdot;r) :=\nod(\cdot ,\cdot;0,r)$ (resp.
$\nod^{*}(\cdot,\cdot;r) :=\nod^{*}(\cdot ,\cdot;0,r)$), consistent to \S\ref{sec:not bas set}. Finally, let
$\nod(g;u,r)=\nod(g,\infty;u,r) $
be the total number of domains lying inside $B_{0}(r)$.

\begin{lemma}[cf. ~\cite{So12}, Lemma 1]
\label{lem:int geom sandwitch euclidian}
Let $\Gamma$ be a closed hypersurface. Then for $0<r<R$, $t>0$,
\begin{equation}
\label{eq:geom int scal invar}
\begin{split}
\frac{1}{\vol (B(r))}\int\limits_{B(R-r)} \nod(\Gamma,t;u,r)du
&\le  \nod(\Gamma,t;R)    \\&\le \frac{1}{\vol (B(r))}\int\limits_{B(R+r)} \nod^{*}(\Gamma,t;u,r)du.
\end{split}
\end{equation}
\end{lemma}

\begin{proof}

We follow along the lines of the proof of ~\cite[Lemma $1$]{So12}:
For a connected component $\gamma$ of $\Gamma$ denote $\mathcal{A}(\gamma)$ to be the area of
the domain having $\gamma$ as its outer boundary.

Let $\gamma$ be any connected component of $\Gamma$.
Define
\begin{equation*}
G_{*}(\gamma) = \bigcap\limits_{v\in \gamma} B_{v}(r) = \{u:\gamma\subseteq B_{u}(r) \}
\end{equation*}
and
\begin{equation*}
G^{*}(\gamma) = \bigcup\limits_{v\in \gamma} \overline{B_{v}(r)} =
\{u:\gamma \cap \overline{B_{u}(r)}\ne \emptyset \}.
\end{equation*}
We have for every $v\in \gamma$, $$G_{*}(\gamma) \subseteq B_{v}(r) \subseteq G^{*}(\gamma),$$ and thus, in particular,
\begin{equation}
\label{eq:vol(G*)<=vol B(r)<=vol(G*)}
\vol(G_{*}(\gamma)) \le \vol(B(r)) \le \vol(G^{*}(\gamma)).
\end{equation}
Summing up \eqref{eq:vol(G*)<=vol B(r)<=vol(G*)} for all connected components $\gamma\subseteq \Gamma$ lying inside $B(R)$,
corresponding to domains of volume at most $t$, we obtain
\begin{equation}
\label{eq:sum vol*bi(g)}
\begin{split}
\sum\limits_{\substack{\gamma:\: \mathcal{A}(\gamma)\le t \\ \gamma \subseteq B(R)}}\vol(G_{*}(\gamma))  &\le
\vol(B(r))\cdot \nod(\Gamma,t;R)  \le
\sum\limits_{\substack{\gamma:\: \mathcal{A}(\gamma)\le t \\ \gamma \subseteq B(R)}} \vol(G^{*}(\gamma)) .
\end{split}
\end{equation}
Writing the volume as an integral
\begin{equation*}
\vol(G_{*}(\gamma)) = \int\limits_{G_{*}(\gamma)}du,
\end{equation*}
and exchanging the order of summation and integration we obtain
\begin{equation}
\label{eq:sum vol G_*(g)*bi(g)lower}
\sum\limits_{\substack{\gamma:\: \mathcal{A}(\gamma)\le t \\ \gamma \subseteq B(R)}}\vol(G_{*}(\gamma)) \ge
\int\limits_{B(R-r)}\left[ \sum\limits_{\substack{\gamma:\: \mathcal{A}(\gamma)\le t\\ \gamma \subseteq B_{u}(r)}} 1 \right] du
= \int\limits_{B(R-r)} \nod (\Gamma,t;u,r)  du,
\end{equation}
since if $u\in B(R-r)$ then $B_{u}(r)\subseteq B(R)$.
Similarly,
\begin{equation}
\label{eq:sum vol G^*(g)*bi(g)upper}
\sum\limits_{\substack{\gamma:\: \mathcal{A}(\gamma)\le t\\ \gamma\subseteq B(R)}} \vol(G^{*}(\gamma))
\le \int\limits_{B(R+r)}\left[
\sum\limits_{\substack{\gamma:\: \mathcal{A}(\gamma)\le t\\ \gamma \cap \overline{B_{u}(r)}\ne\emptyset}}1 \right]du
= \int\limits_{B(R+r)}\nod^{*}(\Gamma,t;u,r)du,
\end{equation}
since if $\gamma\subseteq B(R)$ and for some $u$, $\overline{B_{u}(r)}\cap \gamma \ne \emptyset$, then
necessarily $u\in B(R+r)$.
The statement of the present lemma then follows from substituting \eqref{eq:sum vol G_*(g)*bi(g)lower}
and \eqref{eq:sum vol G^*(g)*bi(g)upper} into \eqref{eq:sum vol*bi(g)}, and dividing both sides by $\vol B(r)$.

\end{proof}

\subsection{Proof of Theorem \ref{thm:scal invar exist}}

\begin{proof}

Let $t>0$ be given, and fix $r>0$; we apply \eqref{eq:geom int scal invar} to $\Gamma=F^{-1}(0)$:
\begin{equation*}
\begin{split}
&\left( 1-\frac{r}{R}\right)^{n}\frac{1}{\vol B(R-r)}\int\limits_{B(R-r)} \frac{\nod (F,t;u,r)}{\vol (B(r))}du
\le \frac{\nod (F,t;R)}{\vol (B(R))}
\\
&\le \left( 1 + \frac{r}{R}\right)^{n}\cdot \frac{1}{\vol B(R+r)} \int\limits_{B(R+r)} \frac{\nod^{*}
(F,t;u,r)}{\vol (B(r))}du
\\&\le \left( 1 + \frac{r}{R}\right)^{n}\frac{1}{\vol B(R+r)}\int\limits_{B(R+r)} \frac{\nod(F,t;u,r)+\mathcal{C}(u,r;t,F)}{\vol (B(r))}du,
\end{split}
\end{equation*}
where $\mathcal{C}(u,r;t,F)$ is the number of domains $\omega\in \Omega(F)$ intersecting $\partial B_{u}(r)$,
of volume
$$
\vol (\omega) \le t,
$$
bounded by the total number of critical points of the restriction $F|_{\partial B_{u}(r)}$
of $F$ to the hypersphere $\partial B_{u}(r)$,
and
$$
\vol(B(R\pm r)) = \vol(B(R))\cdot \left( 1\pm\frac{r}{R} \right)^{n}.
$$
We rewrite
$\nod (F,t;u,r) = \nod (\tau_{u} F,t;r)$, where $\tau_{u}$ is the translation operator
$$
(\tau_{u}F)(x) = F(u+x).
$$
Choose $r$ so small that $(1\pm r/R)^n$ are $\epsilon$-close to $1$, then we have
\begin{equation}
\label{eq:sandwitch for N(R;S,F)}
\begin{split}
&\left( 1-\epsilon\right)\frac{1}{\vol B(R-r)}\int\limits_{B(R-r)} \frac{\nod(\tau_{u}F,t;r)}{\vol (B(r))}du
\le \frac{\nod(F,t;r)}{\vol (B(R))}
\\&\le \left( 1 + \epsilon\right)\frac{1}{\vol B(R+r)}\int\limits_{B(R+r)} \frac{\nod(\tau_{u}F,t;r)+
\mathcal{C}(\tau_{u} F,t;r)}{\vol (B(r))}du.
\end{split}
\end{equation}

Note that for every $r,t$ the functional
$$
F\mapsto \Upsilon_{r;t}(F):=\frac{\nod (F,t;r)}{\vol (B(r))}
$$
(and its translations) is measurable, and since the number of nodal domains is bounded by the number of critical points, this functional has finite expectation by Lemma \ref{lem:Kac Rice crit ball sphere}, part (1).
It then follows from the ergodic theory that both $$\frac{1}{\vol B(R+r)}\int\limits_{B(R+r)}
\frac{\nod_{\cdot} (\tau_{u}F,t;r)}{\vol (B(r))}du$$
and $$\frac{1}{\vol B(R-r)}\int\limits_{B(R-r)} \frac{\nod_{\cdot}(\tau_{u}F,t;r)}{\vol (B(r))}du$$ converge
to (the same) limit in $L^1$
$$\frac{1}{\vol B(R)}\int\limits_{B(R)} \frac{\nod_{\cdot}(\tau_{u}F,t;r)}{\vol (B(r))}du \rightarrow c_{r;t}(\rho):=\E[\Upsilon_{r;t}].$$
(Initially only the existence of the limit is known; it equals the mean value by $L^{1}$-convergence.)

Observe that, if we get rid of $$\mathcal{C}(\tau_{u} F,t;r)$$ on the rhs of \eqref{eq:sandwitch for N(R;S,F)}
then, up to $\pm \epsilon$, both the lhs and the rhs of
\ref{eq:sandwitch for N(R;S,F)} would converge to the same limit $c_{r;t}(\rho)$ (in either $L^{1}$ or a.s.).
We will be able to get rid of $\mathcal{C}(\tau_{u} F,t;r)$ asymptotically for $r$ large; it
will yield that as $r\rightarrow\infty$, we have the limit $c_{r;t}(\rho)\rightarrow c_{t}(\rho)$
satisfying
\begin{equation}
\label{eq:N(t)/B(R)->ct}
\frac{\nod(F,t;R)}{\vol (B(R))} \rightarrow c_{t}(\rho).
\end{equation}
Setting $$\Psi_{\rho}(t) :=\frac{c_{t}(\rho)}{c(\rho)}$$ with $c(\rho)$ the Nazarov-Sodin constant \eqref{eq:tot numb asymp NS}
will ensure that $\Psi_{\rho}(t)$ obeys \eqref{eq:bi/volB(R)->eta} of Theorem \ref{thm:scal invar exist}.
That $\Psi_{\rho}$ is a distribution function will then follow easily: $$t\mapsto c_{t}(\rho)$$ is monotone
nondecreasing by \eqref{eq:N(t)/B(R)->ct}; $\Psi_{\rho}(\infty)=1$ since $c_{\infty}(\rho) = c(\rho)$ is the Nazarov-Sodin constant.

To see that indeed we can get rid of $\mathcal{C}(\tau_{u} F,t;r)$ we use the same ergodic argument as before, now on
$$F\mapsto\frac{\mathcal{C}(F,t;r)}{\vol (B(r))},$$ yielding
the $L^{1}$ limit
$$\frac{1}{\vol B(R+r)}\int\limits_{B(R+r)}\frac{\mathcal{C}(\tau_{u} F,t;r)}{\vol (B(r))}\rightarrow a_{r;t}(\rho),$$
whence
$$
a_{r;t}(\rho) = \E\left[\frac{\mathcal{C}(\tau_{u} F,t;r)}{\vol (B(r))}\right] = O\left(\frac{1}{r}\right)$$
by the second part of Lemma
\ref{lem:Kac Rice crit ball sphere}.
Hence \eqref{eq:sandwitch for N(R;S,F)} implies
\begin{equation*}
\E \left[\left|\frac{\nod (F,t;R)}{\vol (B(R))} -  c_{r;t}(\rho)\right|\right] = O\left(\epsilon+\frac{1}{r}\right);
\end{equation*}
this proves the existence of the limits $$c_{t}(\rho):=\lim\limits_{r\rightarrow\infty}c_{r;t}(\rho),$$
and, as it was mentioned above, $$\Psi_{\rho}(t) :=\frac{c_{t}(\rho)}{c(\rho)},$$ is the distribution function
satisfying the $L^{1}$-convergence \eqref{eq:bi/volB(R)->eta} we were after in Theorem \ref{thm:scal invar exist}.

\end{proof}

\begin{remark}
This proof also shows a.s. convergence but we will not need that for the rest (i.e. the Riemannian case).
\end{remark}

\subsection{Proofs of Theorems \ref{thm:scal invar increase gen} and \ref{thm:scal invar increase RWM}}

\begin{proof}[Proof of Theorem \ref{thm:scal invar increase gen}]

Since the spectral measure satisfies conditions $(\rho1)-(\rho3)$, the existence of the limiting distribution  $\Psi$ is
guaranteed by Theorem \ref{thm:scal invar exist}. We only have to show that $\Psi$ is strictly increasing.

Proof will consist of three steps. First we will construct a deterministic function with nodal domain of given area. After that we will show that the probability that there is a random function with almost the same nodal domain is positive. Finally we show that this implies that the expected density of such nodal domains is positive.

The first step is essentially the same as the argument in \cite{SW}[Proposition 5.2] By condition $(\rho4^{*})$ the interior of the support of $\rho$ is non-empty. Let us assume that $B(\zeta_0,r_0)\subset \supp \rho$. We are going to show that this implies that for every compact $Q$ the functions from $\mathcal{H}(\rho)$ are dense in $C^k(Q)$ for every $k$.

Let $\phi_n$ be a sequence of positive $C^\infty$ functions with supports converging to $\zeta_0$ which approximate $\delta$-function at $\zeta_0$. Functions $\psi_n(\zeta)=\phi_n(\zeta)+\phi_n(-\zeta)$ and all their derivatives belong to $L^2_{\mathrm{sym}}(\rho)$ and hence their Fourier transforms belong to $\mathcal{H}$. The Fourier transform of $\psi_n$ converges to $e^{2\pi i <\zeta_0,x>}$ in any $C^k(Q)$ and its derivatives $\partial^\alpha \psi_n$ converge to  $(-2\pi x)^\alpha e^{2\pi i <\zeta_0,x>}$, where $x^\alpha=\prod x_i^{\alpha_i}$. This proves that we can approximate any polynomial (times a fixed plane wave). Since the polynomials are dense in $C^k(Q)$, we have that $\mathcal{H}$ is dense in $C^k(Q)$.

In particular, for given $t>0$ and $\epsilon>0$ we can take a $C^1$ function such that its nodal domain around the origin is inside some disc $B(0,r)$  and has area $t$. Moreover, we can assume that its gradient is bounded away from zero in some neighbourhood of the boundary of this nodal domain. By approximating this function by a function from $\mathcal{H}$ and applying Lemma \ref{lem:vol pert} we obtain a function $f_t$ from $\mathcal{H}$ such that its nodal domain around the origin has area $\epsilon$-close to $t$ and its gradient is bounded away from zero on the nodal line.

Without loss of generality we may assume that  $\mathcal{H}$ norm of $f_t$ is $1$ and
extend it to an orthonormal basis $\phi_i$.
In this basis $F$ could be written as
$$
F=c_0 f_t +\phi,
$$
where $c_0$ is a standard Gaussian random variable and
$\phi$ is a Gaussian function independent of $c_{0}$ (spanned by all additional basis functions $\phi_i$). Since $f_t$ is non-degenerate, the probability that
$$
\max\{c_{0}|f|,c_{0}|\nabla f_t|\}>1
$$
in some (fixed) neighbourhood of the nodal domain containing the origin is strictly positive. Let $
r$ be such that the nodal domain of $f_t$ is in $B(0,r)$, then if $C^1(B(0,2r))$ norm of $\phi$ is sufficiently small (see Lemma \ref{lem:vol pert} for the precise statement), then the volume of the nodal domain of $F$ is $\epsilon$ close to that of $f_t$. It is a standard fact
(cf. ~\cite{NaSo09}) that the probability that $\phi$ have small norm is strictly positive. This completes the proof that the area of the nodal domain containing the origin is $2\epsilon$ close to $t$ is strictly positive.

We can pack $\R^2$ with infinitely many copies of the disc of radius $2r$ such that this covering is locally uniformly finite. By translation invariance, for each disc the probability that there is a nodal domain of area $2\epsilon$-close to $t$ is strictly positive (independently of the disc). This implies that
$$
\frac{\nod (F,t+2\epsilon;R)-\nod (F,t-2\epsilon;R)}{\vol (B(R))}>c >0
$$
for some positive $c$ and all sufficiently large $R$. This prove that $\Psi(t+2\epsilon)-\Psi(t-2\epsilon)>c$ which means that $\Psi$ is strictly increasing for $t>t_0$.

\end{proof}

\begin{proof}[Proof of Theorem \ref{thm:scal invar increase RWM}]

The proof of this theorem is almost exactly the same as the proof of Theorem \ref{thm:scal invar increase gen}, but the first deterministic step is different. The main difference is that the spectral measure is supported on the sphere, hence its support has empty interior and the space $\mathcal{H}$ is not dense in $C^k(Q)$.

The fact that $\Psi$ vanishes for $t<t_0$ is completely deterministic and follows from Faber--Krahn inequality. Recall that the random plane wave is a solution to $\Delta f+f=0$. Let $\omega$ be a nodal component of a plane wave $F$. Since $F$ has a constant sign in $\omega$ it must be the first eigenfunction of Laplacian in $\omega$. This implies that the first eigenvalue in $\omega$ is $1$. By
Faber--Krahn inequality the volume of $\omega$ is greater then the volume of the ball for which the main eigenvalue is also equal to $1$. It is a well known fact that $|x|^{1-n/2}J_{n/2-1}(|x|)$ is an eigenfunction of Laplacian in $\R^n$ with eigenvalue $1$. It's nodal domain containing the origin is a ball of radius $j_{n/2-1,1}$ and area $t_0$. This proves that the volume of every nodal domain is bounded below by $t_0$.

To prove the last part of the theorem we first claim that for every $t\ge t_0$ there is a domain $\Omega_t$ such that $\vol(\Omega_t)=t$ and the main eigenvalue of the Laplacian is $1$. We define $\Omega_L$ to be the union of the cylinder of length $L$ and radius $j_{n/2-1,1}$ with  two hemispherical caps at  both ends. For $L=0$ this is the ball of the radius $j_{n/2-1,1}$ and as $L\to \infty$ it converges to the infinite cylinder. Since $\Omega_L$ depends on $L$ continuously, the main eigenvalue $\lambda_L$ changes continuously  from $\lambda_0=1$ to $\lambda_\infty=(j_{(n-1)/2-1,1}/j_{n/2-1,1})^2>0$. By rescaling $\Omega_L$ by $\sqrt{\lambda_L}$ we obtain a family of domains such that their main eigenvalue is $1$ and their volumes change continuously from $t_0$ to infinity. Let us relabel this family and say that $\Omega_t$ is the element of the family with volume $t$.

After that we follow the strategy from \cite{CS}. By Whitney approximation theorem we can construct a domain with analytic boundary which is an arbitrary small perturbation of $\Omega_t$. By continuity of the main eigenvalue we can assume that it is arbitrary close to $1$. By rescaling the domain we obtain a domain with analytic boundary, main eigenvalue $1$ and volume  $\epsilon$ close to $t$. Abusing notation we call this domain $\Omega_t$ as well.

Let $g_t>0$ be the main eigenfunction of Laplacian in $\Omega_t$. Since the boundary is analytic, it could be extended to a small neighbourhood $B$ of $\partial\Omega_t$ in such a way that $\inf \|\nabla g_t\|>0$ in a small tubular neighbourhood of $\partial \Omega_t$. By Lax-Malgrange Theorem for every $b>0$  there is a function $f_t$ which is a {\em global} solution to \eqref{eq:Helmholz} such that $\|f_t-g_t\|_{C^1(B)}<b$. If $b$ is sufficiently small, then by Lemma \ref{lem:vol pert} there is a nodal domain of $f_t$ with volume $\epsilon$ close to $\vol(\Omega_t)$. This proves that there is a {\em deterministic} non-degenerate plane wave with a nodal domain such that its volume is $2\epsilon$-close to $t$.
The rest of the proof is exactly the same as in Theorem \ref{thm:scal invar increase gen}.
\end{proof}

\section{Local results}
\label{sec:local results}

\subsection{Local setting and statement of the local result}

Recall that $\gfr_{n,\alpha}$ is the ``scaling limit" of $f_{n,\alpha;T}$ as $T\rightarrow\infty$, around every point $x\in\M$
in the following sense (cf. section \ref{sec:semiclassical limit}). We know that for $T$ large, their  covariance functions $K_{\alpha,T}(u,v)$ and $r_{n,\alpha}$ are, after scaling, asymptotic to each other, uniformly
on a large ball $B(R)\subseteq\R^{n}$. By Lemma \ref{lem:E[|fxL-Fx|<alpha]}they could be coupled so that
$$
\|f_{x;T}-\gfr\|
$$
is arbitrarily small in $C^{1}$-norm on $B(R)$
(or $B(2R)$).

For $u\in \M$, $r>0$ sufficiently small, $t>0$ and $g:\M\rightarrow \R$ a smooth (deterministic) function denote
$\nod(g,t;u,r)$ (resp. $\nod^{*}(g,t;u,r)$)
the number of domains $\omega\in\Omega(g)$ of $g^{-1}(0)$ of volume $$\vol(\omega)\le t$$ lying entirely inside (resp. intersecting)
the geodesic ball $B_{u}(r)\subseteq\M$.
Theorem \ref{thm:loc scal} to follow immediately states that, with the coupling as above, unless $t$ is a discontinuity point
of the limiting volume distribution (i.e. atoms of the $\Psi_{\rho}$ obtained from an application of Theorem \ref{thm:scal invar exist}
on $\gfr_{n,\alpha}$), after appropriate rescaling,
the nodal domain volumes of $f_{\alpha;T}$ are {\em locally} point-wise approximated by the ones of $\gfr_{n,\alpha}$.
Otherwise, if $t_{0}$ is an atom of the said $\Psi_{\rho}$ (existence unlikely, by Conjecture \ref{conj:Psi lim dist gen band lim}),
then the corresponding probability mass may spread in an interval
$(t_{0}-\epsilon,t_{0}+\epsilon)$ ($\epsilon>0$ arbitrarily small) beyond our control
(see the proof of Proposition \ref{prop:bi(R-1,Fx)<=bi(R/L,fL)<=bi(R+1,Fx)} below).

\begin{theorem}[Cf. \cite{So12} Theorem 5]
\label{thm:loc scal}
Let $f(x)=f_{\alpha;T}(x)$ be the random band limited functions \eqref{eq:f band lim def}, $\Psi=\Psi_{n,\alpha}$
the distribution function prescribed by Theorem \ref{thm:scal invar exist} applied on $\gfr_{n,\alpha}$.
Then for every continuity point $t>0$ of $\Psi_{n,\alpha}(\cdot)$, $x\in\M$ and $\epsilon>0$ we have
\begin{equation}
\label{eq:asymp cnt loc c(x)}
\lim\limits_{R\rightarrow\infty}\limsup\limits_{T\rightarrow\infty}
\prob \left\{ \left| \frac{\nod\left(f,\frac{t}{T^{n}};x,\frac{R}{T}\right)}{c(n,\alpha)\cdot\vol (B(R))}-\Psi(t)\right|  > \epsilon\right\}
 = 0,
\end{equation}
where $c(n,\alpha)>0$ is the Nazarov-Sodin constant of $\gfr_{n,\alpha}$.
\end{theorem}

The rest of this section is dedicated to giving a proof of Theorem \ref{thm:loc scal}.

\subsection{Proof of Theorem \ref{thm:loc scal}}
\label{sec:loc scal proof}

We formulate the following proposition that will imply Theorem \ref{thm:loc scal}. The proof is postponed till  section \ref{sec:bi(R-1,Fx)<=bi(R/L,fL)<=bi(R+1,Fx)}
after some preparatory work in section \ref{sec:bi(R-1,Fx)<=bi(R/L,fL)<=bi(R+1,Fx) prep}.

\begin{proposition}
\label{prop:bi(R-1,Fx)<=bi(R/L,fL)<=bi(R+1,Fx)}
Let $x\in \M$, $R>0$ sufficiently big, $\delta,\xi,\eta>0$ small be given, and $t>0$.
Then outside of an event of probability at most $\delta$, for all $T>T_{0}(R,\delta,\xi)$ we have
\begin{equation}
\label{eq:bi(R-1,Fx)<=bi(R/L,fL)<=bi(R+1,Fx)}
\begin{split}
\nod(\gfr_{n,\alpha},t-\xi;R-1) - \eta\cdot R^{n} &\le
\nod \left(f_{x;T},t;R\right) \\&\le \nod (\gfr_{n,\alpha},t+\xi;R+1) + \eta\cdot R^{n}.
\end{split}
\end{equation}
\end{proposition}


\begin{proof}[Proof of Theorem \ref{thm:loc scal} assuming Proposition \ref{prop:bi(R-1,Fx)<=bi(R/L,fL)<=bi(R+1,Fx)}.]

Let $x\in\M$, $\epsilon>0$ be fixed, and $\xi,\eta>0$ be given small numbers.
We observe that by the definition \eqref{eq:fxL def} of the scaled fields $f_{x;T}$ and
\eqref{eq:|vol(D)-vol(Phix(d))|<eps}, for $T$ sufficiently big (depending on $R$) we have
\begin{equation}
\label{eq:nod fxL<=nod fL<=nodfxL}
\nod \left(f_{x;T},t-\xi;R\right) \le \nod \left(f,\frac{t}{T^{n}};x,\frac{R}{T}\right) \le \nod
\left(f_{x;T},t+\xi;R\right)
\end{equation}
Hence Proposition \ref{prop:bi(R-1,Fx)<=bi(R/L,fL)<=bi(R+1,Fx)} together with \eqref{eq:nod fxL<=nod fL<=nodfxL}
holding for $T$ sufficiently big (depending on $R$) imply that for $T$ sufficiently big, outside an event of arbitrarily small probability,
\begin{equation*}
\begin{split}
\nod(\gfr_{n,\alpha},t-2\xi;R-1)-\eta\cdot R^{n} &\le
\nod  \left(f,\frac{t}{T^{n}};x,\frac{R}{T}\right) \\&\le \nod  (\gfr_{n,\alpha},t+2\xi;R+1) + \eta \cdot R^{n}.
\end{split}
\end{equation*}
Choose $$\eta<\frac{\epsilon \cdot \vol B(1)}{2}$$ so that
\begin{equation}
\begin{split}
\label{eq:pr(n-c)<pr(n-c)+pr(n-c)}
&\prob\left\{ \left|\frac{\nod  \left(f,\frac{t}{T^{n}};x,\frac{R}{T}\right)}{c(n,\alpha)\cdot\vol B(R)} -
\Psi(t)\right|>\epsilon\right\} \\&\le
\prob\left\{\left|\frac{\nod  (\gfr_{n,\alpha},t+2\xi;R+1)}{c(n,\alpha)\cdot\vol B(R)} -  \Psi(t) \right| > \epsilon - \eta \right\}
\\&+ \prob\left\{\left|\frac{\nod  (\gfr_{n,\alpha},t-2\xi;R-1)}{c(n,\alpha)\cdot\vol B(R)} -
\Psi(t) \right| > \epsilon -\eta \right\}.
\end{split}
\end{equation}

Now, bearing in mind that $\epsilon-\eta > \frac{\epsilon}{2},$
\begin{equation*}
\begin{split}
&\left\{\left|\frac{\nod  (\gfr_{n,\alpha},t+2\xi;R+1)}{c(n,\alpha)\cdot\vol B(R)} -
\Psi(t) \right| > \epsilon-\eta \right\} \subseteq
\\&\left\{\left|\frac{\nod  (\gfr_{n,\alpha},t+2\xi;R+1)}{c(n,\alpha)\cdot\vol B(R)} -  \Psi(t+2\xi) \right| >
\frac{\epsilon}{2} - (\Psi(t+2\xi)-\Psi(t)) \right\},
\end{split}
\end{equation*}
and thus
\begin{equation}
\label{eq:nod(n-c+)->0}
\begin{split}
&\prob\left\{\left|\frac{\nod  (\gfr_{n,\alpha},t+2\xi;R+1)}{c(n,\alpha)\cdot\vol B(R)} -  \Psi(t)\right| > \epsilon-\eta \right\}
\\&\le \prob\left\{\left|\frac{\nod  (\gfr_{n,\alpha},t+2\xi;R+1)}{\vol B(R)} -  \Psi(t+2\xi) \right| >
\frac{\epsilon}{2} - (\Psi(t+2\xi)-\Psi(t)) \right\}\rightarrow 0
\end{split}
\end{equation}
as $R\rightarrow \infty$ by Theorem \ref{thm:scal invar exist}, and the continuity of $\Psi(\cdot)$ at $t$.
Similarly, as $R\rightarrow \infty$,
\begin{equation}
\label{eq:nod(n-c-)->0}
\prob\left\{\left|\frac{\nod  (\gfr_{n,\alpha},t-2\xi;R-1)}{c(n,\alpha)\vol B(R)} -  \Psi(t) \right| > \epsilon-\eta \right\} \rightarrow 0.
\end{equation}
The above proves that for each $R>0$ sufficiently large there exists
$T_{0}=T_{0}(R)$, such that \eqref{eq:pr(n-c)<pr(n-c)+pr(n-c)} holds for $T>T_{0}$; the latter
could be made arbitrarily small by \eqref{eq:nod(n-c+)->0} and \eqref{eq:nod(n-c-)->0}.
This is precisely the statement of Theorem \ref{thm:loc scal}.

\end{proof}

\subsection{Some preparatory results towards the proof of Proposition \ref{prop:bi(R-1,Fx)<=bi(R/L,fL)<=bi(R+1,Fx)}}

\label{sec:bi(R-1,Fx)<=bi(R/L,fL)<=bi(R+1,Fx) prep}

In course of the proof of Proposition \ref{prop:bi(R-1,Fx)<=bi(R/L,fL)<=bi(R+1,Fx)} we will need to exclude
some exceptional events. Let $\delta>0$ be a small parameter that will control the
probabilities
of the discarded events, $b,\beta>0$ be small parameters that will control the quality of the various approximations,
$\eta>0$ will control the number of discarded domains, and $M,Q>0$ large parameters. Given $R$ and $T$ big we define
\begin{equation*}
\Delta_{1}=\Delta_{1}(R,T;b)= \{ \| f_{x;T}-\gfr_{n,\alpha}\|_{C^{1}(\overline{B}(2R))} \ge b\},
\end{equation*}
\begin{equation*}
\Delta_{2}=\Delta_{2}(R,T;M) = \left\{ \| f_{x,T}\|_{C^{2}(\overline{B}(2R))} \ge M  \right\} ,
\end{equation*}
\begin{equation*}
\Delta_{3}=\Delta_{3}(R;M) = \left\{ \| \gfr_{n,\alpha}\|_{C^{2}(\overline{B}(2R))} \ge M  \right\}
\end{equation*}
and the ``unstable event"
\begin{equation*}
\Delta_{4}(R;\beta) = \left\{ \min\limits_{u\in \overline{B}(2R)}\max \{|\gfr_{n,\alpha}(u)|, |\nabla \gfr_{n,\alpha}(u)|\} \le 2\beta  \right\}.
\end{equation*}
Moreover let $\Delta_{5}$ be the event
\begin{equation*}
\begin{split}
&\Delta_{5}(R;\eta,Q) =
\\&\left\{ \left|\left\{\text{components } \omega \in \Omega(\gfr_{n,\alpha}):\:\omega\subseteq B(2R),\, \vol_{n-1}(\partial\omega)>Q   \right\}\right| >\eta R^{n}\right\}
\end{split}
\end{equation*}
that there is a significant number of nodal domains of $\gfr_{n,\alpha}$ entirely lying in $B(2R)$ whose boundary volume is at least $Q$.
The boundary of such a domain is comprised of a number of nodal components; below we will argue that,
with high probability, each of the boundary components is of bounded volume and their total number is bounded
(see the proof of Lemma \ref{lem:Delta 5 small}).

The following is a simple corollary from Lemma \ref{lem:E[|fxL-Fx|<alpha]}.
\begin{lemma}
\label{lem:Delta1 small}
For a given sufficiently big $R>0$ and small $b,\delta>0$ there exists $T_{0}=T_{0}(R,b,\delta)$ so
that for $T>T_{0}$ the probability of $\Delta_{1}$
\begin{equation}
\label{eq:Delta1 small}
\prob(\Delta_{1}(R,T;b))<\delta
\end{equation}
is arbitrarily small.
\end{lemma}

The following lemma yields a bound on the probabilities $\prob(\Delta_{2})$
and $\prob(\Delta_{3})$.

\begin{lemma}
\label{lem:Delta2-3 small}

\begin{enumerate}
\item
For every $R,\delta>0$ there exists $M=M(R,\delta)>0$ so that
\begin{equation}
\label{eq:Delta3 small}
\prob(\Delta_{3}(R;M)) < \delta.
\end{equation}

\item
For $R,\delta>0$ there exist $M(R,\delta)>0$ and $T_{0}=T_{0}(R)$ so that
\begin{equation}
\label{eq:Delta2 small}
\prob(\Delta_{2}(R,T;M))< \delta
\end{equation}
for all $T>T_{0}$.

\end{enumerate}

\end{lemma}

\begin{proof}

For \eqref{eq:Delta3 small} we may choose $M$ to be
$$
M=\delta^{-1}\E [\| \gfr_{n,\alpha}\|_{C^{2}(\overline{B}(2R))} ],
$$
which is finite by \cite[Theorem $2.1.1$]{Adler-Taylor}. The estimate \eqref{eq:Delta3 small} with this value of $M$
follows from Chebyshev's inequality.

In order to establish \eqref{eq:Delta2 small} we observe that
by \cite{Adler-Taylor}, Theorem $2.2.3$ (``Sudakov-Fernique comparison inequality")
and \eqref{eq:Kalpha asymp}
applied to both $K_{x,T}$ and its derivatives
for all $$M_{1}>\E [\| \gfr_{n,\alpha}\|_{C^{2}(\overline{B}(2R))} ] $$ there exists $T_{0}=T_{0}(R,M_{1})$ such that for all $T>T_{0}$
$$\E \left[ \| f_{x,T}\|_{C^{2}(\overline{B}(2R))}\right] < M_{1}.$$ Hence
\eqref{eq:Delta2 small} with $M=\delta^{-1}M_{1}$ follows from Chebyshev's inequality as before.

\end{proof}

The event $\Delta_{4}$ means that the nodal set of $\gfr_{n,\alpha}$ is relatively unstable. For small $\beta$ the probability  of this event can be made arbitrary small. This is Lemma 7 from \cite{NaSo15}:

\begin{lemma}[Cf. \cite{NaSo15}]
\label{lem:Delta4 small}
For $R>0$  and $\delta>0$ there exist $\beta>0$ such that
\begin{equation}
\label{eq:Delta4 small}
\prob\left(\Delta_{4}(R,\beta) \right)< \delta.
\end{equation}
\end{lemma}

Finally, for $\Delta_{5}$ we have the following bound:

\begin{lemma}
\label{lem:Delta 5 small}
For every $\eta>0$ and $\delta>0$ there exist $Q>0$ and $R_{0}=R_{0}(\eta,\delta$) such that for all $R>R_{0}$
the probability of $\Delta_{5}$ is
\begin{equation}
\label{eq:P(Delta)<delta}
\Delta_{5}(R;\eta,Q) < \delta.
\end{equation}
\end{lemma}

\begin{proof}

In what follows we argue that most of the domains have a bounded number of components
lying in the boundary, and with high probability, each has bounded $n-1$-volume. First, using Nazarov-Sodin's
\eqref{eq:tot numb asymp NS}, there exists a number $A$ sufficiently big so that the probability that
the total number of nodal domains of $\gfr_{n,\alpha}$ entirely lying inside $B(2R)$ is bigger than $A\cdot R^{n}$ is
\begin{equation}
\label{eq:DeltaA<delta/2}
\prob(\Delta_{6}(A)):=\prob\left\{ \nod (\gfr_{n,\alpha};2R) > A\cdot R^{n} \right\}<\delta/2,
\end{equation}
so we may exclude this unlikely event $\Delta_{6}=\Delta_{6}(A)$.

Now we are going to show that the number of boundary components of most of the nodal domains of $\gfr_{n,\alpha}$
lying in $B(2R)$ is bounded; outside of $\Delta_{6}$ there are less than $\eta\cdot R^{n}$. To this end
we use the {\em nesting graph} introduced in ~\cite{SW}. Let $G=(V,E)$ be the graph with the set of vertexes $V$
being the collection of all nodal domains $$V=\{\omega\in\Omega(\gfr_{n,\alpha}):\: \omega\subseteq B(2R)\}$$ of $\gfr_{n,\alpha}$ lying
entirely in $B(2R)$, and an edge $e\in E$ connects between two domains in $V$, if they have a common boundary component.
The graph $G$ is a subgraph of the nesting {\em tree}, same graph with no restriction of the domains
to be contained in $B(2R)$; though $G$ may fail to be a tree, it has no cycles ~\cite[Section 2.4]{SW}; a degree $d(\omega)$
of a vertex $\omega\in V$ in $V$ precisely equals to the number of nodal components lying the boundary of $\omega$
we are to bound. We then have
\begin{equation*}
\sum\limits_{\omega\in V}d(\omega) = 2|E|\le 2(|V|-1) < 2 \nod (\gfr_{n,\alpha};2R),
\end{equation*}
equivalently
\begin{equation*}
\frac{1}{\nod (\gfr_{n,\alpha};2R)}\sum\limits_{\omega\in V}d(\omega)< 2.
\end{equation*}
Hence, by Chebyshev's inequality, outside $\Delta_{6}(A)$ as in \eqref{eq:DeltaA<delta/2},
the number of those $\omega\in V$ with $d(\omega)>L$ is at most
\begin{equation}
\label{eq:num dom > L comp small}
|\{\omega\in V:\: d(\omega)>L\} | \le \frac{2}{L}\cdot \nod (g;2R) \le \frac{2}{L} \cdot AR^{n},
\end{equation}
and below we will choose $L$ sufficiently big so that
\begin{equation}
\label{eq:2A/L < eta/2}
\frac{2A}{L} < \frac{\eta}{2}.
\end{equation}

Next we show that, with high probability, the $(n-1)$-volume of most of the components is bounded.
Let $\Zc_{\gfr_{n,\alpha}}(2R)$ be the nodal volume of $\gfr_{n,\alpha}$ inside $B(2R)$
\begin{equation}
\label{eq:nod vol sum comp}
\Zc_{\gfr_{n,\alpha}}(2R) = \vol_{n-1}(\gfr_{n,\alpha}^{-1}(0) \cap B(2R)) = \sum\limits_{\gamma\in E}\vol_{n-1}(e).
\end{equation}
Then, by a standard application of  Lemma \ref{lem:Kac-Rice precise}, and the stationarity of $\gfr_{n,\alpha}$
the expectation of the nodal volume is given by
\begin{equation}
\label{eq:exp nod vol Kac-Rice}
\E[\Zc_{\gfr_{n,\alpha}}(2R)] = c_{0}\cdot R^{n},
\end{equation}
where $c_{0}=c_{0}(\gfr_{n,\alpha}) >0$ is a positive constant, that could be evaluated explicitly in terms of $n$ and $\alpha$.
Hence, by \eqref{eq:nod vol sum comp} and \eqref{eq:exp nod vol Kac-Rice}, and Chebyshev's inequality, outside of event of probability $\delta/2$
the number of those components with large $(n-1)$-volume is bounded:
\begin{equation}
\label{eq:numb comp big length rare}
\prob(\Delta_{7}(Q/L)):=
\prob\left\{ \left| \left\{\gamma\in E:\: \vol_{n-1}(\gamma) > \frac{Q}{L}\right\}\right|> \frac{\eta}{4}\cdot R^{n}  \right\} < \frac{\delta}{2}
\end{equation}
provided that $Q/L$ is sufficiently big.
Since each component is lying in the boundary of at most two domains $\omega\in V$, it follows that
outside of $\Delta_{7}$ for all but at most $\frac{\eta}{2}\cdot R^{n}$ domains $\omega\in V$, {\em for all}
components $\gamma\in E$ lying in the boundary of $\omega$ we have
\begin{equation}
\vol_{n-1}(\gamma) \le \frac{Q}{L}.
\end{equation}

Now, given $\delta>0$ we choose $A>0$ sufficiently big so that \eqref{eq:DeltaA<delta/2} is satisfied.
This forces a choice of $L$ via \eqref{eq:2A/L < eta/2},
so that the r.h.s. of \eqref{eq:num dom > L comp small} is $<\frac{\eta}{2}\cdot R^{n}$,
and then we take $Q$ sufficiently big so that \eqref{eq:numb comp big length rare} is satisfied
and
\begin{equation*}
\prob(\Delta_{6}\cup\Delta_{7})\le \prob(\Delta_{6})+\prob(\Delta_{7})< \delta.
\end{equation*}
From the above, outside of $\Delta_{6}\cup\Delta_{7}$ for all but $\eta\cdot R^{n}$ nodal domains $\omega\in V$ we have
that the number of boundary components of $\omega$ is $<L$, and the $(n-1)$-volume of {\em each one} of them is bounded
by $\frac{Q}{L}$, and hence the $(n-1)$-volume of $\partial \omega$ is $<Q$ as claimed by this lemma.

\end{proof}

\subsection{Proof of Proposition \ref{prop:bi(R-1,Fx)<=bi(R/L,fL)<=bi(R+1,Fx)}}

\label{sec:bi(R-1,Fx)<=bi(R/L,fL)<=bi(R+1,Fx)}


To prove Proposition \ref{prop:bi(R-1,Fx)<=bi(R/L,fL)<=bi(R+1,Fx)} we need the following Lemma which is the uniform version of an obvious statement that the
volume of a nodal domain depends continuously on the function as long as $0$ is not a critical value.

\begin{lemma}[Cf. ~\cite{So12}, lemmas $6-7$]
\label{lem:vol pert}

Let $R>0$ and $g$ and $h$ be two (deterministic) $C^2$-smooth functions on $B(2R)\subset\R^{n+1}$. We assume that:
\begin{enumerate}
\item For some $\beta>0$ we have
$$\min\limits_{\overline{B}(2R)}\max\{ |g|,|\nabla g | \} > \beta.$$

\item For some $M>0$ the $C^{2}$-norms of both $g$ and $h$ are bounded $$\|g \|_{C^{2}(B(2R))}, \|h \|_{C^{2}(B(2R))} < M.$$

\item We have $$\|g-h\|_{C^{1}(B(R))}<b$$ for some
$b>0$.
\end{enumerate}

Then if $b$ is sufficiently small (depending on $\beta$ and $M$ only)
there exists an injective map $\gamma\mapsto \gamma^{h}$ between connected components
$\gamma\subseteq B(R-1)$ of $g^{-1}(0)$ and connected components $\gamma^{h}\subseteq B(R)$ of $h^{-1}(0)$
with the following properties:

\begin{enumerate}

\item For every $\gamma$ as above the components $\gamma$ and $\gamma^{h}$ are ``uniformly close": there exists a smooth
bijective map $\psi_{\gamma}:\gamma\rightarrow \gamma^{h}$ so that for all $x\in \gamma$ we have
\begin{equation}
\label{eq:psi(gamma)-gamma small}
\|\psi_{\gamma}(x)-x\|_{\infty} = O_{\beta,M}(b)
\end{equation}
with constants involved in the `$O$`-notation depending on $\beta$ and $M$ only.

\item Let $\Gc_{\gamma,\gamma^{h}}$ be the region enclosed between $\gamma$ and $\gamma^{h}$. Its (signed or not) volume satisfies
\begin{equation}
\label{eq:vol pert dom}
|\vol_{n}(\Gc_{\gamma,\gamma^{h}})|=\vol_{n-1}(\gamma)\cdot O_{\beta,M}(b).
\end{equation}


\end{enumerate}

\end{lemma}

The proof of Lemma \ref{lem:vol pert} is postponed until immediately after the proof of
Proposition \ref{prop:bi(R-1,Fx)<=bi(R/L,fL)<=bi(R+1,Fx)}.

\begin{proof}[Proof of Proposition \ref{prop:bi(R-1,Fx)<=bi(R/L,fL)<=bi(R+1,Fx)} assuming Lemma \ref{lem:vol pert}]

We are going to show that the small exceptional event is
$$
\Delta=\bigcup\limits_{i=1}^{5}\Delta_{i}.
$$

Outside $ \Delta$ we have
\begin{equation}
\label{eq:||fxL-Fx||<beta}
\|f_{x;T}-\gfr_{n,\alpha}\|_{C^{1}(\overline{B}(2R))} <b,
\end{equation}
\begin{equation*}
\| f_{x;T}\|_{C^{2}(\overline{B}(2R))} , \| \gfr_{n,\alpha}\|_{C^{2}(\overline{B}(2R))}< M,
\end{equation*}
and assuming $b <\beta$ we also have
\begin{equation}
\label{eq:maxfxL,nabla>beta}
\min\limits_{\overline{B}(2R)}\max\{ |f_{x;T}|,|\nabla f_{x;T} | \}>\beta ; \;
\min\limits_{\overline{B}(2R)}\max\{ |\gfr_{n,\alpha}|,|\nabla \gfr_{n,\alpha} | \} > \beta.
\end{equation}
Hence the conditions of Lemma \ref{lem:vol pert} are satisfied
with $h=f_{x;T}$, $g=\gfr_{n,\alpha}$ (or the other way round) and $b,\beta$ as above.

By Lemma \ref{lem:vol pert} for each nodal domain of $f_{x;T}$ lying in $B(R)$ there is a unique nodal domain of
$\gfr_{n,\alpha}$  which is $O(b)$ close, hence lying in $B(R+1)$. Reversing the roles of $f$ and $\gfr$ we see that the same holds for the nodal domains of $\gfr$. We are going to prove the first inequality of \eqref{eq:bi(R-1,Fx)<=bi(R/L,fL)<=bi(R+1,Fx)},
the proof of the second one being identical.

Let us consider the nodal domains of $\gfr_{n,\alpha}$ that are inside $B(R-1)$ and of area at most $t-\xi$. Their total number is $\nod(\gfr_{n,\alpha},t-\xi,R-1)$. For each of these nodal domains $\Dc$ there is a unique nodal domain $\Dc'$ of $f_{x;T}$ which is $O(b)$ close to and lies in $B(R)$. By equation \eqref{eq:vol pert dom} of Lemma \ref{lem:vol pert} we have
$$
|\vol(\Dc)-\vol(\Dc')|\le O_{\beta,M}(b)\vol_{n-1}(\partial \Dc).
$$

Since we excluded $\Delta_5$, at most $\eta R^n$ of nodal domains $\Dc$ have boundary volume exceeding $Q$.
For all other domains $\vol(\Dc')<t$ provided that $b$ is so small that have $O(b)Q<\xi$. Hence, their number is bounded by $\nod(f_{x;T},t,R)$. This means that
$$
\nod(\gfr_{n,\alpha},t-\xi,R-1) \le \eta R^n +\nod(f_{x;T},t,R)
$$
i.e. the first inequality of \eqref{eq:bi(R-1,Fx)<=bi(R/L,fL)<=bi(R+1,Fx)}.

\vspace{3mm}

Finally we have to show that the probability of $\Delta$ could be made arbitrary small,
using Lemmas \ref{lem:Delta1 small}--\ref{lem:Delta 5 small}. The argument is straightforward but the order in which we have to choose all constants is a bit fiddly.

Let $t,\xi,\delta,\eta>0$ be given. First we use Lemma \ref{lem:Delta 5 small} to choose $R$ and $Q$ sufficiently large so that $\prob(\Delta_5)<\delta$. From now on this value of $R$ is fixed. By Lemma \ref{lem:Delta4 small} there is $\beta>0$ such that $\prob(\Delta_4)<\delta$. By Lemma \ref{lem:Delta2-3 small} we can choose $M$ and $T_0$ large enough so that probabilities of $\Delta_2$ and $\Delta_3$ are bounded by $\delta$. Fix some $0<b<\beta$, applying Lemma \ref{lem:Delta1 small} and (possibly) increasing $T_0$ we make sure that $\prob(\Delta_1)<\delta$ as well.

All in all for the above choice of all constants and for all $T>T_0$ we have that the probability of the exceptional event  $
\Delta:=\cup_{i=1}^{5}\Delta_{i}$ is bounded by $5\delta$. Replacing $\delta$ by $\delta/5$ we complete the proof of the proposition.

\end{proof}

\begin{proof}[Proof of Lemma \ref{lem:vol pert}]


Let $\gamma\subseteq B(R)$ be a connected component of $g^{-1}(0)$ and $x\in\gamma$ be any point on $\gamma$. By implicit function theorem we can
introduce local coordinates $(z,y)$ such that the surface near $x$ can be parameterized as a graph of a smooth function $p(z)$. In other words $\gamma$
around $x$ is given by $\phi(z)=(z,p(z))$ where $p$ is a smooth function on an open domain $U\subset \R^{n-1}$.
Let $$N(x)=N(z)=\nabla g/|\nabla g|$$
be a unit normal vector to $\gamma$ at $x$. Since the second derivatives of $g$ are uniformly bounded, there is a
number $r_0=r_0(\beta,M)>0$ depending on $\beta$ and $M$ only so that
$$
N(x)\cdot \nabla g>\beta/2
$$
on an $r_0$-neighbourhood of $x$ in $\R^{n}$.

Now fix $x\in \gamma$ and consider
$$
\zeta(r)=h(x+rN(x))=g(x+rN(x))+f(x+rN(x)),
$$
where $f=h-g$. Obviously
$$
|\zeta(0)|=|h(x)|<b,
$$
and
$$
\zeta'(r)>\beta/2-b
$$
for  all $|r|<r_0$. For sufficiently small $b$ we then have
\begin{equation}
\label{eq:zeta'>beta/4}
\zeta'>\beta/4
\end{equation}
for all $|r|<r_0$.  This means that  if $4b/\beta<r_0$, which is true for sufficiently small $b$, there is a unique $r$ with $|r|<4b/\beta<r_0$ such
that $\zeta(r)=0$ We denote it by
$r(x)$ or $r(z)$ with $x=\phi(z)$. It is important to notice that $r$ is uniformly bounded in terms of $\beta$ and $M$ and for fixed $\beta$ and $M$ it
is $O(b)$.

All in all, for every $b$ sufficiently small (depending on $\beta$ and $M$ only) the map
$$
x\mapsto x+r(x)N(x)
$$
maps $\gamma$ onto $\gamma^{h}$, moreover,
$$
r(x)=O_{\beta,M}(b).
$$
Since  $\Gc_{\gamma,\gamma^{h}}$ is the domain between two surfaces that are $O_{\beta,M}(b)$ close, the volume of this domain (signed or unsigned) is
bounded by $\vol_{n-1}(\gamma)O_{\beta,M}(b)$. This completes the proof of Lemma  \ref{lem:vol pert}.

\end{proof}

\section{Global results}
\label{sec: global results}

\subsection{
Proofs of the main results: Theorems \ref{thm:limit exist spher harm},   \ref{thm:limit exist band lim}, \ref{thm:alpha<1 Psi increas}, and\ref{thm:alpha=1 Psi increas t>t0}
}
\label{sec:thm lim exist proof}

\begin{notation}[Global notation]

\label{not:glob, |.|pm}

\begin{enumerate}

\item For $y\in\R$ denote $$|y |_{+} = \max\{0, y\},$$ and $$|y |_{-} = \max\{0, -y\},$$
so that $$|\cdot|=|\cdot |_{+}+|\cdot|_{-}.$$

\item Let $\nod (f;t)$ to be the total number
of nodal domains $\omega\in\Omega(f)$ of $f$
of volume $$\vol_{n}(\omega)<t.$$

\end{enumerate}

\end{notation}

\begin{proposition}
\label{prop:glob upper bound}

Let $\{f=f_{\alpha,T}\}_{T>0}$ be the random fields \eqref{eq:f band lim def}, $c(n,\alpha)$ the
Nazarov-Sodin constant of $\gfr_{n,\alpha} $, and $t>0$ a continuity point of
$$\Psi(\cdot)=\Psi_{n,\alpha}(\cdot),$$ as in the formulation of Theorem \ref{thm:loc scal}.
Then the following holds:
\begin{equation}
\label{eq:E[nod+-]->0}
\E\left[\left| \frac{\nod (f,t/T^n)}{c(n,\alpha)\vol_{n}(\M)\cdot T^{n}} - \Psi(t) \right|_{\pm} \right]\rightarrow  0,
\end{equation}
that is, \eqref{eq:E[nod+-]->0} is claimed for both $|\cdot |_{+}$ and $|\cdot|_{-}$,
as in Notation \ref{not:glob, |.|pm}.

\end{proposition}

Proposition \ref{prop:glob upper bound} will be proved in section \ref{sec:proof glob upper bound} for $|\cdot |_{+}$ only, with
the proof for $|\cdot |_{-}$ following along similar, but easier, lines, with no
need to excise the very small and very long domains (see section \ref{sec:very small long comp Riemann}).

\begin{proof}[Proof of Theorem \ref{thm:limit exist band lim} assuming Proposition \ref{prop:glob upper bound}]

Combining both estimates \eqref{eq:E[nod+-]->0} of Proposition \ref{prop:glob upper bound} (i.e, $|\cdot |_{+}$
and $|\cdot |_{-}$) implies that for every $t>0$ continuity point of $\Psi(\cdot)$ we have
\begin{equation}
\label{eq:E[|nod-Psi]->0}
\E\left[\left| \frac{\nod (f,t/T^n)}{c(n,\alpha)\vol(\M)\cdot T^{n}} - \Psi(t) \right| \right]\rightarrow  0,
\end{equation}
which is not quite the same as the statement \eqref{eq:lim exist band lim}
of Theorem \ref{thm:limit exist band lim} as the denominator needs to be replaced by $\nod (f)$
rather than $$c(n,\alpha)\vol(\M)\cdot T^{n}.$$
To this end we use the triangle inequality to write
\begin{equation}
\label{eq:N(t)/N-Psi triangle}
\begin{split}
&\E\left[\left| \frac{\nod (f,t/T^n)}{\nod (f)} - \Psi(t) \right| \right]
\le \E\left[\left| \frac{\nod (f,t/T^n)}{\nod (f)} -
\frac{\nod (f,t/T^n)}{c(n,\alpha)\vol(\M)\cdot T^{n}}  \right| \right] \\&+
\E\left[\left|\frac{\nod (f,t/T^n)}{c(n,\alpha)\vol(\M)\cdot T^{n}} -
\Psi(t) \right| \right].
\end{split}
\end{equation}
Since the second summand of the r.h.s. of \eqref{eq:N(t)/N-Psi triangle} vanishes by \eqref{eq:E[|nod-Psi]->0}
we only need to take care of the first one. We have
\begin{equation}
\label{eq:|N(t)/N-N(t)/cT^n|->0}
\begin{split}
&\E\left[\left| \frac{\nod (f,t/T^n)}{\nod (f)} -
\frac{\nod (f,t/T^n)}{c(n,\alpha)\vol(\M)\cdot T^{n}}  \right| \right]
\\&=  \E\left[\frac{\nod (f,t/T^n)}{c(n,\alpha)\vol(\M)\cdot \nod (f)} \cdot
\left|\frac{\nod (f)}{T^{n}}  - c(n,\alpha)\vol(\M)  \right| \right]\rightarrow 0,
\end{split}
\end{equation}
by \eqref{eq:NS result band lim}, as it is obvious that $$\frac{\nod_{\Omega}(f,t/T^n)}{\nod (f)} \le 1.$$
We finally substitute \eqref{eq:E[|nod-Psi]->0} and \eqref{eq:|N(t)/N-N(t)/cT^n|->0} into
\eqref{eq:N(t)/N-Psi triangle} to yield the statement \eqref{eq:lim exist band lim}
of Theorem \ref{thm:limit exist band lim}.

\end{proof}

\begin{proof}[Proof of Theorem \ref{thm:limit exist spher harm}]

The first assertion of Theorem \ref{thm:limit exist spher harm} is a particular case of
the statement of Theorem \ref{thm:limit exist band lim} with $\M=\Sc^{2}$ the round $2$-sphere, and
$\alpha=1$. The second assertion of Theorem \ref{thm:limit exist spher harm} is the content of
Theorem \ref{thm:alpha=1 Psi increas t>t0} (which, in this case, follows directly from Theorem
\ref{thm:scal invar increase RWM}).

\end{proof}

\begin{proof}[Proofs of Theorems \ref{thm:alpha<1 Psi increas} and \ref{thm:alpha=1 Psi increas t>t0}]

By Theorem Theorems \ref{thm:limit exist band lim} the distribution function is universal and from \eqref{eq:Bnalp clean} we know the spectral measures. In the case $\alpha<1$ it satisfies the axioms $(\rho 1)-(\rho3)$ and $(\rho4^*)$ and Theorem \ref{thm:alpha<1 Psi increas} follows directly from Theorem \ref{thm:scal invar increase gen}. In the case $\alpha=1$ the limiting field is the random plane wave and Theorem \ref{thm:alpha=1 Psi increas t>t0} follows from Theorem \ref{thm:scal invar increase RWM}.

\end{proof}

\subsection{Proof of Proposition \ref{prop:glob upper bound}}

\subsubsection{Excising the very small and very long domains}

\label{sec:very small long comp Riemann}

\begin{definition}

Let $\xi, D>0$ be parameters.

\begin{enumerate}

\item A domains $\omega\in \Omega(f)$ of $f=f_{\alpha;T}$ is called
`$\xi$-{\em small}' if its $n$-dimensional Riemannian volume in $\M$ is $$\vol_{\M}(\omega)<\xi T^{-n}.$$
Let $\nod_{\xi-sm}(f)$ be the total number of $\xi$-small domains (components) of $f$ in $\M$.

\item For $D>0$, a nodal domain $\omega\in\Omega(f)$
is $D$-{\em long} if its diameter is $>D/T$.
Let $\nod_{D-long}(f)$ be the total number of the $D$-long domains of $f$.

\item Given parameters $D,\xi > 0$ a nodal domain $\omega\in\Omega(f)$ is ($D,\xi$)-{\em normal} (or simply normal),
if it is not $\xi$-small nor $D$-long. For $t>0$ let $\nod_{norm}(f,t)$ be the total number of
$(\xi,D)$-normal domains of $f$ of volume $<t,$ and for $x\in\M$, $r>0$ let
$\nod_{norm}(f_{L},t;x,r)$ (resp. $\nod_{norm}^{*}(f_{L},t;x,r)$) be the number of
those contained in $B_{x}(r)$ (resp. intersecting $\overline{{B}_{x}}(r)$). Finally, if $t$ is omitted, then
it is assumed to be infinite $t=\infty$, i.e. we no restriction on the domain volume is imposed.

\end{enumerate}

\end{definition}

By the definition of normal domains, for every $t\in \R\cup\{\infty\}$ we have
\begin{equation}
\label{eq:Nnorm*(r)<=N(r+D/L)}
\nod_{norm}^{*}(f,t;x,r) \le \nod_{norm}\left(f,t;x,r+\frac{D}{T}\right)
\end{equation}
(as we discarded the very long ovals), and
\begin{equation}
\label{eq:Nnorm<=delta^{-1}vol(B)}
\nod_{norm}(f;x,r) \le \xi^{-1}T^{n} \vol_{\M}(B_{x}(r)),
\end{equation}
by the natural volume estimate.

\begin{lemma}[Cf. ~\cite{So12} Lemma $9$, see also ~\cite{SW} for a more detailed proof]
\label{lem:E[Nxi-sm/L^{n}]->0}
There exist constants $c,C>0$ so that we have the following estimate on the number of $\xi$-small components:
\begin{equation*}
\limsup\limits_{T\rightarrow\infty}\frac{\E[\nod_{\xi-sm}(f_{\alpha;T})]}{T^{n}} \le C\cdot \xi^{c}.
\end{equation*}
\end{lemma}

\begin{lemma}[Cf. ~\cite{So12} Lemma $8$]
\label{lem:E[ND-long/L^{n}]->0}
There exists a constant
$C>0$ such that the following bound holds for the number of $D$-long components:
\begin{equation*}
\limsup\limits_{T\rightarrow\infty}\frac{\E[\nod_{D-long}(f_{\alpha;T})]}{T^{n}} \le C\cdot \frac{1}{D}.
\end{equation*}
\end{lemma}

\subsubsection{Proof of Proposition \ref{prop:glob upper bound}}

\label{sec:proof glob upper bound}

\begin{proposition}
\label{prop:glob upper bound norm tree}
Let $\xi, D>0$ be fixed, and $$\Psi=\Psi_{n,\alpha}.$$ Then for every $t>0$ continuity of $\Psi(\cdot)$ the following holds:
\begin{equation}
\label{eq:E[Gnorm]+ -> 0}
\E\left[\left| \frac{\nod_{norm}(f_{\alpha;T},t/T^{n})}{c(n,\alpha)\vol(\M)\cdot T^{n}} - \Psi(t) \right|_{+} \right]\rightarrow  0.
\end{equation}
\end{proposition}

The proof of Proposition \ref{prop:glob upper bound norm tree} will be given in section \ref{sec:proof prop glob norm}.

\begin{proof}[Proof of Proposition \ref{prop:glob upper bound} assuming Proposition \ref{prop:glob upper bound norm tree}]

The estimate \eqref{eq:E[nod+-]->0} for $|\cdot |_{+}$ follows directly from \eqref{eq:E[Gnorm]+ -> 0},
lemmas \ref{lem:E[ND-long/L^{n}]->0}-\ref{lem:E[Nxi-sm/L^{n}]->0}, and the triangle inequality for $|\cdot |_{+}$.
The proof of \eqref{eq:E[nod+-]->0} for $|\cdot |_{-}$ follows along the same lines, except that in this case
we do not need to excise the very small and very long domains (which makes the proof somewhat simpler).

\end{proof}

\subsection{Proof of Proposition \ref{prop:glob upper bound norm tree}}
\label{sec:proof prop glob norm}

We need to formulate a couple of auxiliary lemmas.

\begin{lemma}[cf. ~\cite{So12} Lemma $1$, and Lemma \ref{lem:int geom sandwitch euclidian}
in the scale-invariant case]
\label{lem:int geom sand Riemann}

Given $\epsilon > 0$, there exists $\eta>0$ such that for every $r<\eta$, $t>0$
\begin{equation}
\label{eq:Nnorm<=(1+eps)*int loc}
\begin{split}
&(1-\epsilon)\int\limits_{\M}\frac{\nod_{norm}(f_{\alpha;T},t;x,r)}{\vol(B(r))}dx
\le \nod_{norm}(f_{\alpha;T},t) \\&\le (1+\epsilon)\int\limits_{\M}\frac{\nod_{norm}^{*}(f_{\alpha;T},t;x,r)}{\vol(B(r))}dx
\end{split}
\end{equation}

\end{lemma}

The proof of Lemma \ref{lem:int geom sand Riemann} is very similar to the one of Lemma \ref{lem:int geom sandwitch euclidian} and is omitted here.

\begin{proof}[Proof of Proposition \ref{prop:glob upper bound norm tree}]

For convenience in course of this proof we will assume that $\M$ is unit volume
\begin{equation}
\label{eq:vol(M)=1}
\vol(\M)=1,
\end{equation}
and recall that $f=f_{\alpha;T}$ are the band-limited random fields \eqref{eq:f band lim def}, and that we use
the shorthand $$\Psi(t)=\Psi_{n,\alpha}(t)$$ as in the formulation of Proposition \ref{prop:glob upper bound norm tree}.
Let $\epsilon>0$ be a small number.
To bound the l.h.s. of \eqref{eq:E[Gnorm]+ -> 0} we let $R>0$ to be sufficiently big,
so that both $R/T<\eta$ as in Lemma \ref{lem:int geom sand Riemann}, and $D/R$ is sufficiently small, so that
the following holds
\begin{eqnarray}
\label{eq:volMB(r)/volB(r)<1+eps}
\left|\frac{\vol_{\M}(B_{x}(R/T))}{\vol(B(R/T))} -1\right| &<& \epsilon,
\\  \frac{\vol(B(R+D))}{\vol(B(R))} &<& 1+\epsilon \nonumber
\end{eqnarray}
uniformly for $x\in\M$.

Now apply Lemma
\ref{lem:int geom sand Riemann} with $r=\frac{R}{ T}$ and $t$ replaced by $\frac{t}{T^{n}}$;
by the triangle inequality for $|\cdot |_{+}$ we have
\begin{equation}
\label{eq:sand apply nodT}
\begin{split}
&\E\left[\left| \frac{\nod_{norm}(f_{\alpha;T},t/T^{n})}{c(n,\alpha)\cdot T^{n}} - \Psi(t) \right|_{+} \right]
\\&\le \E\left[ \int\limits_{\M}\left|(1+2\epsilon)\frac{\nod_{norm}^{*}(f_{\alpha;T},t/T^{n};x,R/T)}{c(n,\alpha)\cdot\vol B(R+D)}
-\Psi(t)\right|_{+}dx \right]
\\&\le
\E\left[ \int\limits_{\M}\left|\frac{\nod_{norm}(f_{\alpha;T},t/T^{n};x,(R+D)/T)}{c(n,\alpha)\cdot\vol B(R+D)}-\Psi(t)\right|_{+}dx \right]
\\&+O\left(\epsilon\cdot \int\limits_{\M}\frac{\E[\nod_{norm}(f_{\alpha;T};x,(R+D)/T)]}{\vol B(R+D)}dx \right),
\end{split}
\end{equation}
by \eqref{eq:Nnorm*(r)<=N(r+D/L)} and \eqref{eq:volMB(r)/volB(r)<1+eps}.
Observe that the integrand $$\frac{\E[\nod_{norm}(f_{\alpha;T};x,(R+D)/T)]}{\vol B(R+D)}$$ is uniformly bounded by
Lemma \ref{lem:Kac-Rice band lim bnd unif}. Hence \eqref{eq:sand apply nodT} is
\begin{equation*}
\begin{split}
&\E\left[\left| \frac{\nod_{norm}(f_{\alpha;T},t/T^{n})}{c(n,\alpha)\cdot T^{n}} - \Psi(t) \right|_{+} \right]
\\&\le\E\left[\int\limits_{\M}\left| \frac{\nod_{norm}(f_{\alpha;T},t/T^{n};x,(R+D)/T)}{c(n,\alpha)\cdot
\vol B(R+D)}-\Psi(t)\right|_{+}dx\right] +
O(\epsilon).
\end{split}
\end{equation*}
It is then sufficient to prove that
\begin{equation}
\begin{split}
\label{eq:int Omega,M N(S,x)-cS(x)}
&\E\left[\int\limits_{\M} \left|\frac{\nod_{norm}(f_{\alpha;T},t/T^{n};x,(R+D)/T)}{c(n,\alpha)\cdot \vol B(R+D)}-\Psi(t)\right|_{+}dx\right] \\&=
\int\limits_{\Delta}\int\limits_{\M} \left|\frac{\nod_{norm}(f_{\alpha;T},t/T^{n};x,(R+D)/T)}{c(n,\alpha)\cdot \vol B(R+D)}-\Psi(t)\right|_{+}dx
d\Pc(\omega)
\rightarrow 0,
\end{split}
\end{equation}
where $\Delta$ is the underlying probability space, and $\Pc$ is the probability measure on $\Delta$.
Now consider the event
$$\Delta_{T,t;x,R} = \left\{  \left| \frac{\nod (f_{\alpha;T},t/T^{n};x,R/T)}{c(n,\alpha)\cdot
\vol B(R+D)}-\Psi(t)\right|>\epsilon\right\}.$$
Then, recalling the assumption of Proposition \ref{prop:glob upper bound norm tree} on $t$ (i.e. that $\Psi(\cdot)$ is continuous at $t$),
Theorem \ref{thm:loc scal} implies that for all $x\in\M$
\begin{equation}
\label{eq:P(Omega)->0}
\lim\limits_{R\rightarrow\infty}\limsup\limits_{T\rightarrow\infty} \Pc(\Delta_{T,t;x,R} ) =0.
\end{equation}
We claim that the above implies that there exists a sequence $\{R_{j}\}_{j\rightarrow\infty}$ of values $R=R_{j}\rightarrow\infty$ so that
the limit \eqref{eq:P(Omega)->0} is almost uniform w.r.t. $x\in X$, that is, for every $\rho>0$
there exists $\M_{\rho} \subseteq \M$ of volume $$\vol \M_{\rho}>1-\rho,$$ such that
\begin{equation}
\label{eq:P(Omega) Egorov}
\varliminf\limits_{R\rightarrow\infty}\varlimsup\limits_{T\rightarrow\infty} \sup\limits_{x\in\M_{\rho}}\Pc(\Delta_{T,t;x,R+D}) =0.
\end{equation}
To see \eqref{eq:P(Omega) Egorov} we first apply an Egorov-type theorem on the limit in \eqref{eq:P(Omega)->0} w.r.t.
$R\rightarrow\infty$: working with the sets $$E_{n,k} =
\bigcup\limits_{R>n \text{ integer}}\left\{x\in \M:\: \Pc(\Delta_{T,t;x,R+D}) > \frac{1}{k} \text{ for } T=T_{j}\rightarrow\infty\right\}$$
yields that for some $M_{\rho}$ with $$\vol(M_{\rho})>1-\frac{\rho}{2}$$ we have
\begin{equation*}
\lim\limits_{R\rightarrow\infty}\sup\limits_{x\in\M_{\rho}}\varlimsup\limits_{T\rightarrow\infty} \Pc(\Delta_{T,t;x,R+D}) =0;
\end{equation*}
this is not quite the same as the claimed result \eqref{eq:P(Omega) Egorov}, as the order of $\sup\limits_{x\in\M_{\rho}}$
and the $\limsup$ w.r.t. $T\rightarrow\infty$ is wrong.
We use an Egorov-type argument once again, w.r.t. the limit $\varlimsup\limits_{T\rightarrow\infty}$
to mollify this.
Fix an integer $r>0$, and let $R=R(r)>0$ sufficiently big so that
\begin{equation}
\label{eq:sup(limsup)<1/r}
\sup\limits_{x\in\M_{\rho}}\varlimsup\limits_{T\rightarrow\infty} \Pc(\Delta_{T,t;x,R+D})<\frac{1}{r}.
\end{equation}
Define the monotone decreasing sequence of sets
$$F_{m} = \bigcup\limits_{T>m}\left\{x\in \M_{\rho}:\: \Pc(\Delta_{T,t;x,R+D}) > \frac{2}{r} \right\}.$$
Since, by \eqref{eq:sup(limsup)<1/r}, $$\bigcap\limits_{m\ge 1} F_{m}=\emptyset,$$
we may find $m=m(r)$ sufficiently
big so that $\vol(F_{m(r)})<\frac{\rho}{2^{r+1}}$. Therefore the claimed result \eqref{eq:P(Omega) Egorov} holds on
$$M_{\rho}\setminus \bigcup\limits_{r\ge 1}F_{m(r)},$$ i.e.
further excising the set $$\bigcup\limits_{r\ge 1}F_{m(r)} $$ of volume $<\frac{\rho}{2}$ from $\M_{\rho}$.

We then write the integral \eqref{eq:int Omega,M N(S,x)-cS(x)} as
\begin{equation}
\label{eq:int OmegaM [NV-eta]+}
\begin{split}
&\int\limits_{\Delta}\int\limits_{\M} \left|\frac{\nod_{norm}(f_{\alpha;T},t/T^{n};x,(R+D)/T)}{c(n,\alpha)\cdot \vol
B((R+D))}-\Psi(t)\right|_{+}dx d\Pc (\omega)\\&=\int\limits_{\M}\int\limits_{\Delta_{T,t;x,R+D}}
+\int\limits_{\M}\int\limits_{\Delta\setminus\Delta_{T,t;x,R+D}}.
\end{split}
\end{equation}
First, on $\Delta\setminus\Delta_{T,t;x,R+D}$, the integrand of \eqref{eq:int OmegaM [NV-eta]+} is
\begin{equation*}
\begin{split}
&\left|\frac{\nod_{norm}(f_{\alpha;T},t/T^{n};x,(R+D)/T)}{c(n,\alpha)\cdot \vol B((R+D)/T)}-\Psi(t)\right|_{+} \\&\le
\left|\frac{\nod (f_{\alpha;T},t/T^{n};x,(R+D)/T)}{c(n,\alpha)\cdot \vol B((R+D)/T)}-\Psi(t)\right|   \le \epsilon,
\end{split}
\end{equation*}
and hence the contribution of this range is
\begin{equation}
\label{eq:int Nnorm /Omega}
\begin{split}
&\int\limits_{\M}\int\limits_{\Delta\setminus\Delta_{T,t;x,R+D}}
\left|\frac{\nod_{norm}(f_{\alpha;T},t/T^{n};x,(R+D)/T)}{c(n,\alpha)\cdot \vol B((R+D)/T)}-\Psi(t)\right|_{+}dx d\Pc (\omega)
\\&\le \int\limits_{\M}\int\limits_{\Delta\setminus\Delta_{T,t;x,R+D}}\epsilon dxd\Pc(\omega) \le \epsilon.
\end{split}
\end{equation}

On $\Delta_{T,t;x,R+D}$ we use the volume estimate \eqref{eq:Nnorm<=delta^{-1}vol(B)} yielding uniformly on $x\in\M$
\begin{equation}
\label{eq:NV triv estimate Omega}
\begin{split}
&\int\limits_{\Delta_{T,t;x,R+D}}
\left|\frac{\nod_{norm}(f_{\alpha;T},t/T^{n};x,(R+D)/T)}{c(n,\alpha)\cdot\vol B(R+D)}-\Psi(t)\right|_{+} d\Pc (\omega)
\\&\le \int\limits_{\Delta_{T,t;x,R+D}}
\left|\frac{\nod_{norm}(f_{\alpha;T},t/T^{n};x,(R+D)/T)}{c(n,\alpha)\cdot\vol B(R+D)}\right|_{+}d\Pc (\omega)
\\&\le \int\limits_{\Delta_{T,t;x,R+D}}\xi^{-1}T^{n}\frac{\vol_{\M}(B_{x}((R+D)/T))}{c(n,\alpha)\cdot\vol B(R+D)}d\Pc (\omega)  \\&\le
(1+\epsilon)\xi^{-1}\Pc(\Delta_{T,t;x,R+D}).
\end{split}
\end{equation}
Similarly to the above, uniformly for $\omega\in\Delta$
\begin{equation}
\label{eq:NV triv estimate M/Meta}
\int\limits_{\M\setminus \M_{\rho}}
\left|\frac{\nod_{norm}(f_{\alpha;T},t/T^{n};x,(R+D)/T)}{c(n,\alpha)\cdot \vol B(R+D)}-\Psi(t)\right|_{+} dx \le (1+\epsilon)\xi^{-1}\rho.
\end{equation}

The uniform estimates \eqref{eq:NV triv estimate Omega} and \eqref{eq:NV triv estimate M/Meta} imply that
\begin{equation*}
\begin{split}
&\int\limits_{\M}\int\limits_{\Delta_{L,t;x,R+D}}
\left|\frac{\nod_{norm}(f_{\alpha;T},t/T^{n};x,(R+D)/T)}{c(n,\alpha)\cdot \vol B(R+D)}-\Psi(t)\right|_{+}
dx\Pc(\omega) \\&\le (1+\epsilon)\xi^{-1}
(\sup\limits_{x\in \M_{\rho}}\Pc(\Delta_{T,t;x,R+D}) + \rho).
\end{split}
\end{equation*}
Upon substituting the latter estimate and \eqref{eq:int Nnorm /Omega} into \eqref{eq:int OmegaM [NV-eta]+},
and then to the integral \eqref{eq:int Omega,M N(S,x)-cS(x)}, we finally obtain
\begin{equation*}
\begin{split}
&\E\left|\int\limits_{\M} \frac{\nod_{norm}(f_{\alpha;T},t/T^{n};x,(R+D)/T)}{c(n,\alpha)\cdot\vol B(R+D)}-\Psi(t)\right|_{+}dx
\\&\le \epsilon + (1+\epsilon)\xi^{-1}
(\sup\limits_{x\in \M_{\rho}}\Pc(\Delta_{T,t;x,R+D}) + \rho),
\end{split}
\end{equation*}
which could be made arbitrarily small for each sufficiently small choice of $\xi$ excising the very small
components, and using \eqref{eq:P(Omega) Egorov}. This concludes the proof of \eqref{eq:int Omega,M N(S,x)-cS(x)},
sufficient to yield the conclusion of Proposition \ref{prop:glob upper bound norm tree}.

\end{proof}

\section{Final remark: Volumes of the nodal components}
\label{sec: final remarks}

\subsection{Volume distribution of boundary components}

We would like to point out that the methods we are using are rather general and with minimal change
one may prove other results. These in particular include similar results about the $(n-1)$ dimensional volumes
of nodal components or boundaries of nodal domains instead of the $n$ dimensional volume of the nodal domains.
The analogue of Theorem \ref{thm:scal invar exist} is proved in exactly the same way using Kac-Rice for the bound and ergodic theory for the existence of the limit. Going from the planar case to the Riemannian one is also straighforward.
All these results could be rewritten line-by-line, with the definition of $\nod$ essentially the only change.

The only difference is in the results stating that the limit functions are strictly increasing i.e. Theorems \ref{thm:scal invar increase gen} and \ref{thm:scal invar increase RWM}.
The main idea is still the same: we have to create a deterministic example with a nodal component of (approximately)
given size, and then show that small perturbations do not significantly alter its size with positive probability.
The first part follows along the lines precisely as before. For $\alpha<1$ the Hilbert space $\mathcal H$ is dense in any $C^k(Q)$, hence we can construct any non-degenerate example. For $\alpha=1$, the same family of domains works for the nodal sets, the minimal volume will be the volume of the surface volume of the sphere of radius $j_{n/2-1,1}$ instead of the the $n$ dimensional volume of the corresponding ball. (This follows from the isoperimetrical inequality.) The only real difference is in the Lemma \ref{lem:vol pert}: controlling the $C^1$ norm is not sufficient in order to control the change of the boundary volume. But it is not very difficult to show that if the perturbation has a small
$C^2$-norm, then the ratio of boundary volumes of the original and perturbed domains is $1+O(b)$ where the constant in $O(b)$ depend only on the $C^2$ norms.

\subsection{Coupling between domain boundary volumes}

Finally, we point that it is possible to consider the {\em joint} distribution of $$(\vol_n(\omega),\vol_{n-1}(\partial \omega))$$ where $\omega$ is a nodal domain. The existence of the scaling limit is established via an identical argument to the above. Constructing a deterministic example in the case $\alpha<1$ also follows along the lines of the argument in this manuscript. This would give that the limiting distribution $\Psi(t,l)$ is strictly increasing in both $t$ and $l$ for all $t>0$ and $l>l_0(t)$, where $l_0(t)$ is the surface area of a sphere with volume $t$ (follows from the isoperimetric inequality).

The case $\alpha=1$ is more complicated. From the isoperimetric inequality, for each $t>t_0$, the size of the boundary must be $\ge l_0(t)$. But it is clear that this is not the best possible lower bound, as otherwise the domain would be a ball and its main eigenvalue would be larger than $1$ unless $t=t_0$. On the other hand, there {\em exists} a deterministic infimum of the boundary volume over all domains that have volume $t\ge t_0$ and the main eigenvalue $1$. We denote this infimum by $l_1(t)$. There exists a domain for which the boundary is {\em arbitrary close} to $l_1(t)$. Next we attach to this domain a very thin tube and at the end of this tube we attach a very thin plate, see Figure \ref{fig:coupling}. This only slightly perturbs the volume, does not alter the main eigenvalue significantly, but increases the boundary volume a lot. After that we take a small perturbation to make the domain real analytic and use Lax-Malgrange to approximate the main eigenfunction by a plane wave. This way we can construct a nodal domain with eigenvalue $1$, volume arbitrary close to $t\ge t_0$ and boundary volume arbitrary close to $l\ge l_1(t)$.

\begin{figure}
\label{fig:coupling}
\includegraphics[width=9cm]{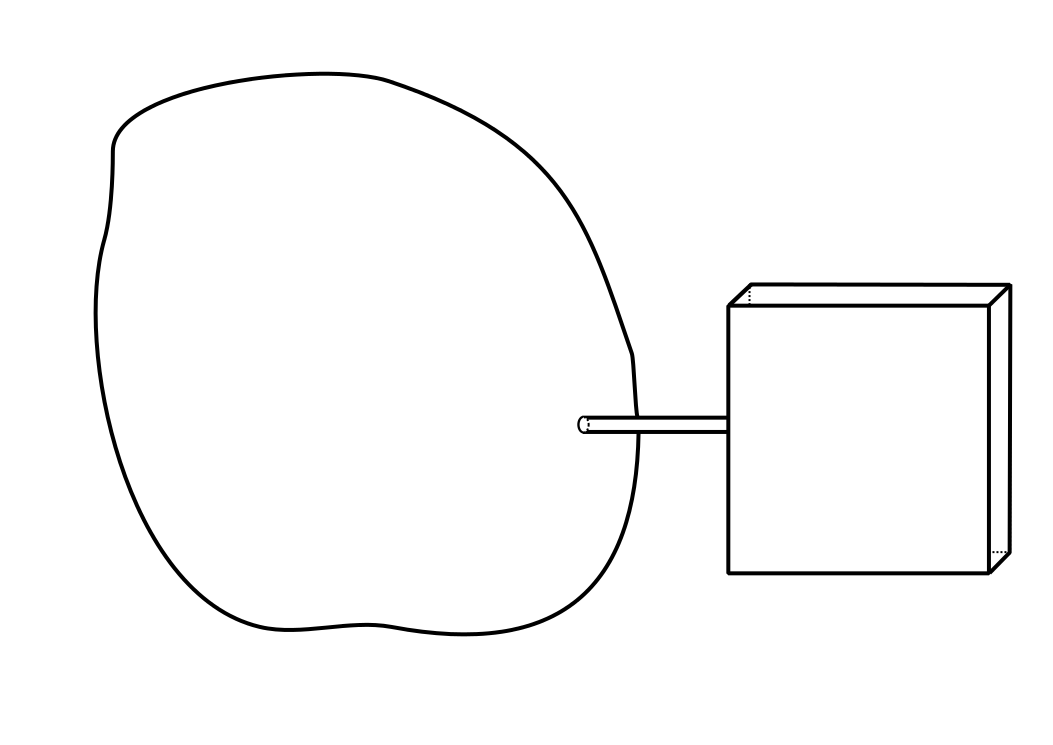}
\caption{Connecting a thin plate by a thin tube.}
\end{figure}

\bibliography{nodal}
\bibliographystyle{abbrv}

\end{document}